\documentclass[a4paper,11pt,reqno]{amsart} 

\pdfoutput=1

\usepackage{fullpage}

\usepackage[sort&compress,comma,square,numbers]{natbib}

\usepackage[indent,skip=.2\baselineskip]{parskip}

\usepackage[T1]{fontenc}
\usepackage[utf8]{inputenc}

\usepackage{lmodern}

\usepackage{microtype} 

\usepackage{url}

\usepackage{amsmath}
\usepackage{amssymb}

\usepackage[dvipsnames,table]{xcolor}

\usepackage[unicode,psdextra,pdfusetitle,
bookmarks, bookmarksdepth=2, colorlinks=true, linkcolor=blue!50!black, citecolor=Black!70, urlcolor=blue!50!black]{hyperref}

\usepackage{bookmark}
\hypersetup{bookmarksnumbered}

\usepackage{amsthm}
\usepackage[nameinlink]{cleveref}

\newtheorem{theorem}{Theorem}[section]
\newtheorem{lemma}[theorem]{Lemma}
\newtheorem{proposition}[theorem]{Proposition}

\theoremstyle{definition}
\newtheorem{definition}[theorem]{Definition}
\newtheorem{example}[theorem]{Example}

\newtheorem{remark}[theorem]{Remark}

\newtheorem{theoremalphabetic}{Theorem}

\AtBeginEnvironment{example}{%
  \pushQED{\qed}%
}
\AtEndEnvironment{example}{\popQED\endexample}

\crefname{ex}{Example}{Examples}
\crefname{thm}{Theorem}{Theorems} 
\crefname{lem}{Lemma}{Lemmas}
\crefname{prop}{Proposition}{Propositions}
\crefname{cor}{Corollary}{Corollaries} 
\crefname{con}{Conjecture}{Conjectures} 
\crefname{def}{Definition}{Definitions}
\Crefname{algocf}{Algorithm}{Algorithms}
\crefname{rmk}{Remark}{Remarks}
\crefname{thmalph}{Theorem}{Theorems}

\numberwithin{equation}{section}
\numberwithin{figure}{section}

\usepackage[linesnumbered,ruled,noend]{algorithm2e}
\SetArgSty{textnormal}

\DeclareTextFontCommand{\bfemph}{\bfseries\em}
\newcommand{\term}{\bfemph}

\usepackage{graphicx}
\usepackage[font=scriptsize,labelfont=bf]{caption}
\usepackage{subcaption}
\usepackage{float}
\usepackage{wrapfig}

\usepackage{multirow}
\usepackage{array} 

\usepackage{mathrsfs} 

\usepackage{bm}

\usepackage{bbm} 

\usepackage{enumitem}

\definecolor{specialblue}{rgb}{0.07,0.35,0.57}
\definecolor{specialred}{rgb}{0.7,0.2,0.2}

\usepackage[perpage]{footmisc}

\usepackage{pgf,tikz} 
\usetikzlibrary{arrows}
\usetikzlibrary{decorations.markings}
\usetikzlibrary{matrix}
\usetikzlibrary{cd}
\usetikzlibrary{babel}
\usetikzlibrary{positioning}
\usetikzlibrary{shapes}
\usetikzlibrary{arrows.meta,petri}

\usepackage{pgfplots}
\pgfplotsset{compat=1.8}

\newcommand{\from}{{\colon}}

\DeclareMathOperator{\sign}{sign} 
\DeclareMathOperator{\row}{row} 
\DeclareMathOperator{\rk}{rk} 
\DeclareMathOperator{\im}{im} 
\DeclareMathOperator{\spn}{span} 
\DeclareMathOperator{\diag}{diag} 
\DeclareMathOperator{\supp}{supp} 

\newcommand{\1}{\mathbbm{1}} 

\renewcommand{\S}{\mathcal{S}}

\newcommand{\Y}{\mathcal{Y}}
\newcommand{\X}{\mathcal{X}}
\newcommand{\K}{\mathcal{K}}

\newcommand{\T}{\mathcal{T}}

\newcommand{\Z}{\mathcal{Z}}
\newcommand{\U}{\mathcal{U}}

\newcommand{\ZZ}{{\mathbb{Z}}}
\newcommand{\QQ}{{\mathbb{Q}}}
\newcommand{\RR}{{\mathbb{R}}}
\newcommand{\CC}{{\mathbb{C}}}

\newcommand{\GG}{{\mathbb{G}}}

\newcommand{\kk}{\mathbbm{k}}

\newcommand{\VV}{\mathbb{V}}

\newcommand{\I}{\mathcal{I}}
\newcommand{\cL}{\mathcal{L}}

\renewcommand{\k}{\kappa}

\usepackage{longtable}
\usepackage{booktabs}

\usepackage[version=3]{mhchem}

\usepackage[colorinlistoftodos,prependcaption,
    textsize=footnotesize,  
    linecolor=gray, bordercolor=gray, backgroundcolor=gray!25, 
    textcolor=black]{todonotes}

\usepackage{subfiles}

\title{Toric invariance of vertically parametrized systems}

\author{Elisenda Feliu and 
Oskar Henriksson}

\begin{document}

\begin{abstract}
We consider the problem of deciding whether the solution sets of a parametrized polynomial system are toric in the sense that they admit a monomial parametrization.
We focus on vertically parametrized systems, which are sparse systems where we allow linear dependencies between coefficients in front of the same monomial.
We give a matroid-theoretic characterization of the maximal-dimensional torus for which all solution sets are invariant under componentwise multiplication. 
Building on this, we provide necessary conditions and sufficient conditions for when the solution sets are unions of finitely many or a unique coset  of this torus.
The motivation of this work comes from the theory of reaction networks, where toric structure of the steady state system substantially simplifies the determination of multistationarity; here, we show that this is also the case for absolute concentration robustness and steady state invariants. 
\end{abstract}

\maketitle

\section{Introduction}
\label{sec:intro}

Toric varieties are central objects in algebraic geometry, and appear naturally in many applications, including polynomial system solving \cite{sottile:book,Telen2022}, statistics \cite{GeigerMeekSturmfels2006},  phylogenetics \cite{SturmfelsSullivant2005}, and chemical reaction network theory \cite{Craciun2009toric}. Much is known about the geometry of toric varieties \cite{Cox2011toric},
but effectively deciding whether a given variety is toric from an implicit description remains a hard problem, which has recently been approached from the point of view of, e.g., symbolic computation \cite{Grigoriev2020toricity} and Lie theory \cite{MarajPal2026,KahleVill2025}.

This paper focuses on deciding whether the \emph{positive} zero sets 
of a \emph{parametric} polynomial system have toric structure.
More precisely, 
for a parametric system $F\in\RR[\k_1,\ldots,\k_m,x_1^\pm,\ldots,x_n^\pm]^s$ (with parameters $\k$ and variables $x$)
we consider the set $\Z_{>0}$ of parameters $\k\in \RR^m_{>0}$ for which  $\VV_{>0}(F_\k)$ is nonempty. We say that $F$
is \term{toric} over $\RR_{>0}$   if 
there is an  exponent matrix $A\in\ZZ^{d\times n}$ such that  for all $\k\in \Z_{>0}$, $\VV_{>0}(F_\k)$ admits a monomial parametrization of the form
\[ \RR_{>0}^d\to\VV_{>0}(F_\k),\quad t\mapsto \alpha_\k\circ t^A\, \]
for some $\alpha_\k\in\RR_{>0}^n$ depending on $\k$, and where $\circ$ denotes componentwise multiplication.
In this case, we view $\VV_{>0}(F_\k)$ as a coset of the multiplicative subgroup $\T_A^{>0}=\{t^A:t\in\RR_{>0}^d\}\subseteq\RR^n_{>0}$ and write \[\VV_{>0}(F_\k)=\alpha_\k\circ\T_A^{>0}.\] Detecting {whether a system is toric in this sense}
can be formulated as a quantifier elimination problem \cite{Rahkooy2021parametric}, which gives a detection algorithm, albeit at a substantial computational cost. Another general but computationally intense approach is considered in \cite{sadeghimanesh2019grobner,conradi2019total,rahkooy2021comprehensive}, which give sufficient conditions in terms of Gröbner bases and comprehensive Gröbner systems. 

Here, we focus on the special case when $F$ is \term{vertically parametrized} in the language of \cite{HelminckRen2025,FeliuHenrikssonPascual2025vertical}, in the sense that it can be written as
\[F=C(\k\circ x^M)\, , \]
where each row of $C\in\RR^{s\times m}$ encodes a linear combination of $m$ monomials with exponents given by the columns of $M\in\ZZ^{n\times m}$ and scaled by the parameters $\k=(\k_1,\ldots,\k_m)$.

\begin{example}
\label{ex:IDH}
As a simple running example throughout the paper, we will use the following  vertically parametrized system: 
\[
F=\left[{\small\begin{array}{l}
-
\kappa_{{1}}x_{{1}}x_{{2}}+\kappa_{{2}}x_{{3}}+\kappa_{{3}}x_{{3}}\\
-\kappa_{{1}}x_{{1}}x_{{2}}+\kappa_{{2}}x_{{3}}+\kappa_{{6}}x_{{5}}\\
\kappa_{{4}}x_{{3}}x_{{4}}-\kappa_{{5}}x_{{5}}-\kappa_{{6}}x_{{5}}
\end{array}}\right],\quad  C={\small\begin{bmatrix}
-1&1&1&0&0&0\\
-1&1&0&0&0&1\\ 
0&0&0&1&-1&-1
\end{bmatrix}},\quad M= {\small\begin{bmatrix} 1&0&0&0&0&0\\ 1&0&0&0
&0&0\\ 0&1&1&1&0&0\\0&0&0&1&0&0\\ 0&0&0&0&1&1\end{bmatrix}}.
\]
It is not hard to see that {the system is toric} over $\RR_{>0}$ for the exponent matrix
\[A={\small\begin{bmatrix}1 & 0 & 1 & 0 & 1\\
0 & 1 & 1 & 0 & 1\end{bmatrix}}\  \in \ZZ^{2\times 5}, \]
as, for each $\k\in\RR^6_{>0}$, the positive zero locus admits the parametrization
\[\RR^2_{>0}\to\VV_{>0}(F_\k)\,,\quad
(t_1,t_2)\mapsto \left(t_1,  t_2, \tfrac{\kappa_{1}}{\kappa_{2} + \kappa_{3}}\, t_1 t_2, \tfrac{\kappa_{3} (\kappa_{5} + \kappa_{6})}{\kappa_{4} \kappa_{6}}, \tfrac{\kappa_{1} \kappa_{3} }{\kappa_{6} (\kappa_{2} + \kappa_{3})}\,t_1 t_2\right).\qedhere \] 
\end{example}

\smallskip

Vertically parametrized systems describe the steady states of reaction networks and the critical points of multivariate polynomials. They include sparse polynomial systems with fixed support, but the framework also allows for linear dependencies among the coefficients of the same monomial.  
Previous work on vertically parametrized systems has addressed the generic root count \cite{HelminckRen2025} and tropical homotopies for polynomial system solving \cite{tropicalhomotopy}. Algebraic-geometric properties such as dimension and nondegeneracy were treated in \cite{FeliuHenrikssonPascual2025vertical}, where it was shown that generic solvability, in the sense that
$\Z_{>0}$ has nonempty interior, is equivalent to the linear algebra condition
\begin{equation}\label{eq:Zp}
    \rk(C\diag(w)M^\top)=\rk(C) \:\:\text{for some $w\in\ker(C)$\,. }
\end{equation}

For reaction networks, knowing that the steady state system is toric over $\RR_{>0}$ is especially relevant, since it simplifies the analysis of key problems such as multistationarity. In this context, several \emph{sufficient} conditions to {guarantee that the steady state system $F$ is toric are known}. This includes binomiality of the ideal $\langle F_\k\rangle$ for all $\k$ \cite{Millan2012toricsteadystates,conradi2015detecting}, as well as the concept of \emph{complex-balanced equilibria} \cite{Horn1972general,Craciun2009toric}, and the related notions of deficiency \cite{Feinberg1995class,Boros2012multiple} and disguised toric dynamical systems \cite{Brustenga2022disguised}. 

In this work, we {first investigate}  
the \emph{necessary} condition that $F$ is \term{invariant}  with respect to $A$, meaning that $\VV_{>0}(F_\k)$ is invariant under componentwise multiplication by $\T_A^{>0}$ for each $\k\in\RR_{>0}^m$. 
Equivalently, each $\VV_{>0}(F_\k)$ is a union of (possibly zero or infinitely many) cosets of $\T_A^{>0}$.
{This property can be seen as the positive analog of being a \emph{T-variety} over the complex numbers \cite{AltmannIltenPetersenSussVollmert2012}.} 
When the union consists of finitely many cosets, we say that {the system is \term{locally toric}}. 
 See \Cref{fig:overview} for an illustration of these notions. 

\begin{figure}[b]
    \centering
    \begin{subfigure}{0.25\textwidth}
    \centering
    \includegraphics[height=9em]{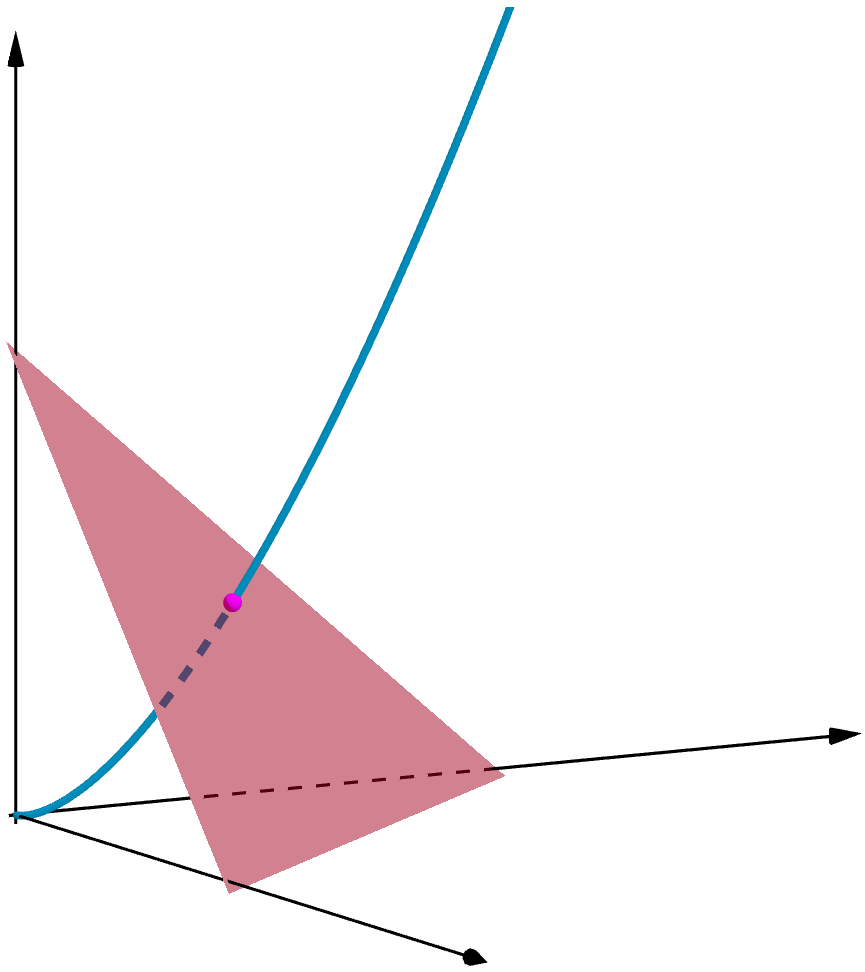}
    \caption{Single coset}
    \end{subfigure}
    \hspace{3em}
    \begin{subfigure}{0.25\textwidth}
    \centering
    \includegraphics[height=9em]{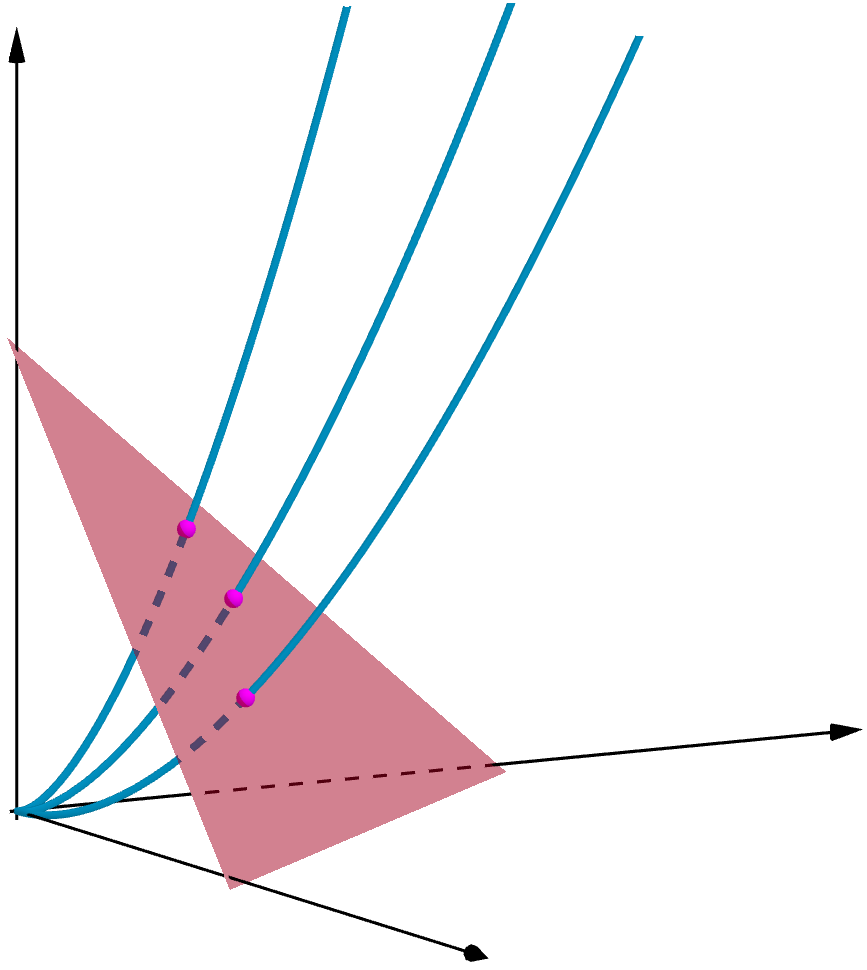}
    \caption{Finitely many cosets}
    \end{subfigure}
    \hspace{3em}
    \begin{subfigure}{0.25\textwidth}
    \centering
    \includegraphics[height=9em]{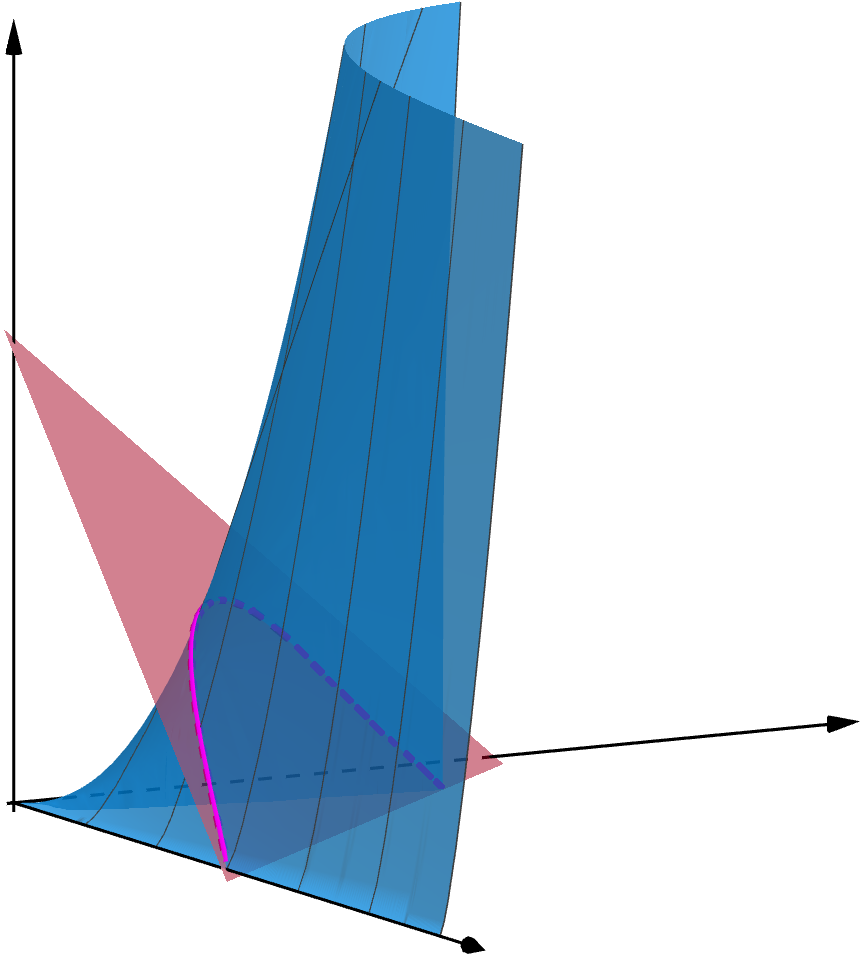}
    \caption{Infinitely many cosets}
    \end{subfigure}
    \caption{Three  invariant varieties (in  blue)  with respect to the exponent matrix $A=[\,1 \:\: 2 \:\: 3\,]$. The cosets are in bijection with the intersections with a parallel translate of the plane $\ker(A)$ (in red).}
    \label{fig:overview}
\end{figure}

Our approach to detecting and ruling out toric structure has the following three steps, the theoretical foundations of which are summarized in parts (i)-(iii) of \Cref{thmA:main_result}.
First, we determine the unique (up to row operations) maximal-rank matrix $A$ for which $\VV_{>0}(F_\k)$ is invariant for all $\k$ through the linear algebra condition in (i). Secondly, in part (ii), we characterize when the dimension of $\T_A^{>0}$ and of $\VV_{>0}(F_\k)$ generically agree, to determine whether $\VV_{>0}(F_\k)$ is generically a finite union of $\T_A^{>0}$-cosets. Importantly, part (i) and (ii) together give a way to rule out $F$ being toric. Finally, we derive a range of sufficient conditions for $\VV_{>0}(F_\k)$ to consist of a single coset. These conditions come from the observation that, by Birch's Theorem, each coset intersects each parallel translate of $\ker(A)$ precisely once, thus giving rise to the \emph{coset counting system} appearing in the disjoint union in (iii).

\begin{theoremalphabetic}[\Cref{thm:invariance_lifts_to_closure,thm:characterization_of_invariance,thm:characterization_generic_local_toricity} and \Cref{prop:coset_counting}]
\label{thmA:main_result}
Let $F=C(\k\circ x^M)$ be a vertically parametrized system defined by $M=[\,M_1\,M_2\cdots M_m\,]\in\ZZ^{n\times m}$ and $C\in\RR^{s\times m}$   with $\ker(C)\cap \RR_{>0}^m\neq\varnothing$, and let $A\in\ZZ^{d\times n}$.

\begin{enumerate}
\item[(i)] (Invariance) 
If $C'$ is the reduced row echelon form of $C$, then $F$ is invariant with respect to $A$ if and only if 
\[
    A(M_{i}-M_j)=0 \quad 
    \text{for all $i,j\in\supp(v)$ and all  rows $v$ of $C'$}.\]
\item[(ii)]   (Locally toric systems) 
Suppose $F$ is invariant with respect to $A$ and that \eqref{eq:Zp} holds. 
Then $\VV_{>0}(F_{\k})$ is a union of finitely many $\T_A^{>0}$-cosets for generic $\k\in \Z_{>0}$  if and only if 
\[{\rk(C)+\rk(A)=n\,.}\]
\item[(iii)] (Counting cosets) Suppose $F$ is invariant with respect to $A$. Then, for each $\k\in\RR^m_{>0}$ and each $b\in A(\RR^n_{>0})$, it holds that
\[\VV_{>0}(F_\k) = \bigsqcup_{\alpha \in \VV_{>0}(F_{\k},\, Ax-b)} \alpha\circ \T_A^{>0}\,.\]
\end{enumerate}
\end{theoremalphabetic}

Invariance is the topic of \Cref{sec:toricity,sec:invariance}. We prove part (i) by characterizing the stronger notion of invariance of the zero locus over $\CC^*$, and showing that invariance of the positive zero locus is equivalent to invariance of the complex zero locus. This is not true for general parametric systems; see \Cref{ex:reductionC}.  Notably, part (i) shows that invariance only depends on the \emph{column matroid} of $C$ (rather than the specific entries of $C$) and on the affine span of subsets of columns of $M$ determined by the \emph{connected components} of this matroid. In particular, invariance can only arise 
if the system, after putting $C$ in reduced row echelon form, has the property that each polynomial has exponent vectors belonging to a parallel translate of $\row(A)^\perp$.  This generalizes the well-known fact that a polynomial is quasihomogeneous if and only if its Newton polytope is contained in an affine hyperplane \cite[Ch.~6.1]{gelfand1994discriminants}. In informal terms, this means that invariance reflects that the system ``effectively depends on fewer than $n$ variables'' as discussed in \Cref{rem:true_dim}. This is not the case for arbitrary parametric systems; see \Cref{ex:true_dim}.

Part (ii) is the topic of \Cref{sec:local} and the coset counting system of part (iii) is discussed in \Cref{sec:counting}. Here, we also derive bounds on the number of cosets and sufficient conditions for $F$ to be toric.  These conditions are gathered in \Cref{alg:summary} and implemented in our Julia package \texttt{ToricVerticalSystems.jl}.   

In \Cref{sec:crnt_perspectives}, we study consequences of invariance and (local) toric structure in the context of reaction network theory and use our approach to provide new mathematical insights to previously established results. After recalling the basic terminology in \Cref{subsec:crnt_intro}, we study in  \Cref{subsec:crnt_applications} the implications of toric structure for three central concepts in reaction network theory: \emph{multistationarity}, \emph{absolute concentration robustness},  and \emph{steady state invariants}. In general, these are computationally  challenging properties to detect or rule out, but with (locally) toric structure for a known exponent matrix $A$,
the analysis simplifies substantially.

In \Cref{subsec:intermediates}, we introduce a model-reduction technique, based on the concept of intermediates {from \cite{Feliu2013intermediates,sadeghimanesh:multi}}, that preserves invariance. In \Cref{subsec:odebase}, we
apply our algorithms to biochemically relevant networks from the database ODEbase \cite{odebase2022} as a case study. The latter shows that our conditions are often enough to conclusively 
{determine whether the steady state system is toric}
for realistic networks.  

Finally, in \Cref{subsec:other_flavors}, we connect our results to {more specialized notions of toric structure} in the  literature: quasithermostatic networks (and the related notion of \emph{toric dynamical systems} of \cite{Craciun2009toric}), and networks with binomial steady state ideal (also known as  \emph{toric steady states} in \cite{Millan2012toricsteadystates}). In particular,  \Cref{thmA:main_result} provides an easy-to-check \emph{necessary} condition for both properties to hold in a region of parameter space with nonempty Euclidean interior.

\subsection*{Notation and conventions}
The cardinality of a set $S$ is denoted by $\# S$.
For   $n\in\ZZ_{>0}$, we  let $[n]=\{1,\ldots,n\}$. For a field $\kk$, we denote $\kk\setminus \{0\}$ by $\kk^*$. 
We write $\circ\from\kk^n\times\kk^n\to\kk^n$ for componentwise multiplication. 
For  $A\in\RR^{n\times m}$, we write $A_i$ for the $i$th column, and $A_{i*}$ for the $i$th row. 
The \emph{support} of $v\in\RR^n$ is the set $\supp(v)=\{ i\in [n] : v_i\neq 0\}$. 
For $x\in\RR^n_{>0}$, we let $x^{-1}$ be defined componentwise, and $x^A\in\RR^m$ be defined by $(x^A)_j=x_1^{a_{1j}}\cdots x_n^{a_{nj}}$ for $j\in [m]$.

\subsection*{Acknowledgements}
We thank two anonymous referees for helpful feedback on earlier versions of the manuscript.
This work has been supported by the Novo Nordisk Foundation, with grant reference  NNF20OC0065582, and by the European Union under the
Grant Agreement number 101044561, POSALG. Views and opinions expressed are those of the
authors only and do not necessarily reflect those of the European Union or European Research
Council (ERC). Neither the European Union nor ERC can be held responsible for them. 

\section{Vertically parametrized systems}
Throughout, we work with Laurent polynomials with coefficients in $\kk\in\{\RR,\CC\}$. 
We consider \term{vertically parametrized systems} (or \term{vertical systems} for short), which are parametric  systems of the form  
\begin{equation}
\label{eq:NkB}
F=C(\k\circ x^M)\in \kk[\k_1,\ldots,\k_m,x_1^\pm,\ldots,x_n^\pm]^s,
\end{equation}
consisting of $s\leq n$ polynomials with parameters $\k=(\kappa_1,\ldots,\kappa_m)$ and variables 
$x=(x_1,\ldots,x_n)$, encoded by a coefficient matrix $ C\in \kk^{s\times m}$ and an exponent matrix  $M \in \ZZ^{n\times m}$.
The component\-wise product $\k\circ x^M$ indicates that the monomial encoded by the $i$th column of $M$ is scaled by $\k_i$, while the rows of $C$ give  linear combinations of the scaled monomials. An important feature is that $F$ is linear in the parameters and that each parameter always accompanies the same monomial (though a monomial is accompanied by different parameters if $M$ has repeated columns).
We also consider \term{augmented vertically parametrized systems} of the form
\[\left(C(\k\circ x^M),\, Lx-b\right)\in \kk[\k,b,x^\pm]^{s+d}, \qquad s+d \leq n\,,\]
where we also allow for $d\geq 0$ affine linear equations, encoded by a coefficient matrix $L\in\kk^{d\times n}$ 
and parametric constant terms $b=(b_1,\ldots,b_d)$. Geometrically, this corresponds to intersecting the variety given by the vertical system $C (\k \circ x^M)$ by a parallel translate of $\ker(L)$. 

For an augmented vertical system $F$, we denote the specialization at $(\k,b)\in\kk^{m+d}$ by 
\[F_{\k,b}=F(\k,b,\cdot) \in\kk[x^\pm]^{s+d}\, . \]
{Although our primary interest is in the positive zeros of $F$, it is convenient to work in a more general setting.  We will therefore consider the set of zeros over a multiplicative group of scalars $\GG\in \{ \RR_{>0},\RR^*,\CC^*\}$, and the corresponding solvability locus, denoted by } 
\begin{equation*}
\VV_{\GG}(F_{\k,b})=\{x \in\GG^n : F_{\k,b}(x)=0 \}\,, \qquad 
\Z_{\GG}=\{(\k,b)  \in\GG^m\times \kk^d : \VV_{\GG}(F_{\k,b}) \neq \varnothing \}\,.
\end{equation*}
We implicitly consider the ground field to be $\kk = \RR$ if $\GG$ is $\RR^*$ or $\RR_{>0}$, and   $\kk=\CC$ if $\GG=\CC^*$.  We will often abbreviate ``$\RR_{>0}$'' by ``${>\!0}$'' in superscripts and subscripts.

Key results from  \cite{FeliuHenrikssonPascual2025vertical} on generic dimension and nondegeneracy for augmented vertical systems  are reviewed next, as they play an important role later on. 

For an augmented vertical system $F\in \kk[\k,b,x^\pm]^{s+d}$, we  say that a zero $x\in \GG^n$ of $F_{\k,b}$    is \term{nondegenerate}  if the Jacobian $J_{F_{\k,b}}(x)$ has rank $s+d$.  A  nondegenerate zero of $F_{\k,b}$ is in particular a nonsingular point of  $\VV_{\CC^*}(F_{\k,b})$. 
We recall that $\dim(\VV_{\CC^*}(F_{\k,b}))\geq n-s-d$ for all $(\k,b)\in(\CC^*)^m\times \CC^d$ such that $\VV_{\CC^*}(F_{\k,b})\neq\varnothing$, 
and if in addition, all irreducible components contain a nondegenerate zero, then  $\VV_{\GG}(F_{\k,b})$ has pure dimension $n-s-d$  (cf. \cite[§9.6,~Thm.~9]{cox2015ideals}). 

A summary of the main conclusions from \cite{FeliuHenrikssonPascual2025vertical} is given in the next proposition. 

\begin{proposition}\label{prop:vertical}
Let $\GG \in\{ \RR_{>0},\RR^*,\CC^*\}$ and $F\in \kk[\k,b,x^\pm]^{s+d}$ be an augmented vertical system with defining matrices  $M\in\ZZ^{n\times m}$, $C\in \kk^{s\times m}$ of rank $s$ with $\ker(C)\cap \GG^m \neq \varnothing$, and $L\in\kk^{d\times n}$.  
Consider the rank condition
\begin{equation}\label{eq:nondeg_conditionA}
\rk\begin{bmatrix}C\diag(w)M^\top\diag(h)\\ L \end{bmatrix}= s+d \quad  \text{for some } w\in \ker(C) \text{ and }h\in (\CC^*)^n.
\end{equation}
Then one of the following two scenarios occurs:
\begin{itemize}
\item[(i)] If \eqref{eq:nondeg_conditionA} holds, then 
$\Z_{\GG}$ has nonempty Euclidean interior. Furthermore, 
there exists a nonempty Zariski open subset $\U$ of $\Z_{\GG}$ such that for all $(\k,b)\in \U$, it holds that 
\[ \dim(\VV_{\CC^*}(F_{\k,b}))=\dim(\VV_{\GG}(F_{\k,b}))=n-s-d\]  
and that all zeros of $F_{\k,b}$ in $(\CC^*)^n$ are nondegenerate. If in addition \eqref{eq:nondeg_conditionA} holds for all $w\in \ker(C)\cap \GG^m$ and $h\in \GG^n$, then we can take $\U=\Z_\GG$.
\item[(ii)] If \eqref{eq:nondeg_conditionA} does not hold, then $\Z_{\GG}$ is contained in a hypersurface, and, for all $(\k,b)\in \Z_{\GG}$, it holds that $\dim(\VV_{\CC^*}(F_{\k,b}))>n-s-d$ and that all zeros of $F_{\k,b}$ in $(\CC^*)^n$ are degenerate. 
\end{itemize}
\end{proposition}

\begin{proof}
This follows from Theorem 3.7, together with Propositions 2.11, 3.2 and 3.11, and Remark~3.6 in   \cite{FeliuHenrikssonPascual2025vertical}. 
\end{proof}

In view of \Cref{prop:vertical}, we say that 
an augmented vertical system $F\in \kk[\k,b,x^\pm]^{s+d}$   is \term{nondegenerate over $\GG$} if $\ker(C)\cap \GG^m \neq \varnothing$ and \eqref{eq:nondeg_conditionA} holds. This is equivalent to $F$ having a nondegenerate zero in $\GG$ for some parameter choice and also to 
 $\Z_{\GG}$ being Zariski dense in $\CC^m$. 
 If \eqref{eq:nondeg_conditionA} does not hold, then we say that $F$ is \term{degenerate}. 
Observe that nondegeneracy of $F$ over $\GG$ is equivalent to nondegeneracy over $\CC^*$, as long as $\ker(C)\cap \GG^m \neq \varnothing$. 

When $F$ is a vertical system, condition  \eqref{eq:nondeg_conditionA} reduces to
\begin{equation}\label{eq:nondeg_condition}
\rk(C \diag(w) M^\top)=s \quad \text{for some }w\in \ker(C)\,.
\end{equation}

\begin{example}
For the vertical system in \Cref{ex:IDH}
and $w=(2,1,1,2,1,1) \in \ker(C)\cap \RR^6_{>0}$, 
the matrix in \eqref{eq:nondeg_condition} becomes
\[C \diag(w) M^\top ={\small \left[\begin{array}{ccccc}
-3 & -3 & 3 & 0 & 0 
\\
 -3 & -3 & 1 & 0 & 2
\\
 0 & 0 & 3 & 3 & -3 
\end{array}\right]},\]
which has  rank $3$. Hence $F$ is nondegenerate, $\Z_{>0}$ has nonempty Euclidean interior, and $\VV_{>0}(F_\k)$ has  dimension $2$ for generic $\k\in \Z_{>0}$. 
\end{example}

\begin{remark}[Freely parametrized systems]
\label{rmk:freely_parametrized_systems}
As discussed in \cite[§3.6]{FeliuHenrikssonPascual2025vertical}, vertical systems  include \term{freely parametrized systems}, obtained by fixing the support and letting all coefficients vary freely. Given support sets $\mathcal{S}_1,\ldots,\mathcal{S}_s\subseteq\ZZ^n$, the corresponding freely parametrized system is the vertical system given by
\begin{equation*}
    C = \begin{bmatrix}
    C_1 & \dots & 0 \\ \vdots & \ddots & \vdots \\ 0 & \dots & C_s
    \end{bmatrix}\quad \text{with}\quad C_i=[ 1 \ \cdots\ 1] \in \CC^{1\times \#\mathcal{S}_i}, \quad\text{and}\quad M= [ M_1 \ \cdots \ M_s]\,,
 \end{equation*}  
where the columns $M_i\in \ZZ^{n\times \#\mathcal{S}_i}$ are the elements of $\mathcal{S}_i$ (in some fixed order). 
Restricting to $\GG=\RR_{>0}$ allows us to consider systems with fixed support and coefficients with fixed sign by specifying the signs of the $C_i$ in the construction above (coefficients of free sign are included by repeating the monomial with opposite signs). For instance, the vertical system
\begin{equation}
\label{eq:triangle}
F=(\kappa_1-\kappa_2)x_1^3x_2^2+\kappa_3x_2^4-2\kappa_4x_1^6\quad \in\RR[\kappa_1,\kappa_2,\kappa_3,\kappa_4,x_1^\pm,x_2^\pm]
\end{equation}
consisting of a single polynomial
can be thought of as a generic system with support $x_1^3x_2^2$, $x_2^4$ and $x_1^6$, where the coefficient of $x_1^3x_2^2$ may  take arbitrary signs, and the coefficients of $x_2^4$ and $x_1^6$ are fixed to $+$ and $-$, respectively. (We will revisit this system in \Cref{ex:trianglenetwork_revisited}.)
\end{remark}

\section{{Toric invariance} and cosets}
\label{sec:toricity}

In this section, we define the various forms of {toric invariance} that we study in this work.

\begin{definition}\label{def:postorus}
The  \term{$\GG$-torus} and \term{torus}  associated with $A\in\ZZ^{d\times n}$ are, respectively, 
\[\T_A^\GG=\{t^A: t \in \GG^d\}\subseteq\GG^n, \qquad \T_A=\T_A^{\CC^*} =\{t^A: t \in (\CC^*)^d\}\subseteq\CC^n  \,.\]
\end{definition}

We view $\T_A^{\GG}$ as a multiplicative subgroup of $\GG^n$ and consider the cosets $\alpha\circ\T_A^{\GG}$ for $\alpha\in\GG^n$.
Note that the exponent matrix $A$ can be replaced by any other matrix with the same row lattice $\row_{\ZZ}(A)$ (and over $\CC^*$ and $\RR_{>0}$, it suffices that the row spans over $\QQ$ agree). 
Also note that over $\CC^*$, a coset $\alpha\circ\T_A$ is the same as a \emph{very affine scaled toric variety}. 
 
We now use cosets to define the three notions of invariance illustrated in \Cref{fig:overview}.

\begin{definition}
\label{def:toricity}
For $A\in\ZZ^{d\times n}$, $\GG\in \{ \RR_{>0},\RR^*,\CC^*\}$, and a set $X\subseteq\GG^n$, we say that:
\begin{itemize}
    \item  $X$  is \term{$\T_A$-invariant over $\GG$} if $X\circ\T_A^{\GG}\subseteq X$ 
    (i.e., $x\circ t^A\in X$ for all $x\in X$ and  $ t\in\GG^d$). In this case, $X$ is a union of $\T_A^{\GG}$-cosets. We denote the set of these cosets by $X/\T_A^{\GG}$. 
    \item $X$ is \term{locally $\T_A$-toric over $\GG$} if $X$ is  
    $\T_A$-invariant over $\GG$ and $1\leq\#(X/\T_A^{\GG})<\infty$.
    \item $X$ is \term{$\T_A$-toric over $\GG$}  if it is $\T_A$-invariant over $\GG$ and  $\#(X/\T_A^{\GG})=1$.
\end{itemize}
\end{definition}

\begin{definition}
\label{def:toricity_vertical}
For $A\in\ZZ^{d\times n}$, $\GG\in \{ \RR_{>0},\RR^*,\CC^*\}$, and a vertical system $F\in \kk[\k,x^\pm]^s$, we say:
\begin{itemize}
    \item $F$ is \term{$\T_A$-invariant over $\GG$} if $\VV_{\GG}(F_\k)$ is $\T_A$-invariant over $\GG$ for all $\k\in \GG^m$, that is, 
\begin{equation}\label{eq:toric_explained}
F_\k(x\circ t^A)=0, \qquad \text{for all}\quad \k\in \GG^m,\  x\in\VV_{\GG}(F_\k), \text{ and } t\in\GG^d\,.
\end{equation}

\item $F$ is \term{(generically) (locally) $\T_A$-toric over $\GG$}  if  $F$ is nondegenerate over $\GG$ and $\VV_{\GG}(F_\k)$ is (locally) $\T_A$-toric over $\GG$ (generically) for  $\k\in \Z_\GG$.  
\end{itemize}
If $\GG$ is omitted, it is implicitly assumed that $\GG=\CC^*$. 
 \end{definition}

\begin{example}

The vertical system $F$ in \Cref{ex:IDH} satisfies $\Z_{>0}=\RR^6_{>0}$ and is $\T_A$-toric over $\RR_{>0}$ for the matrix $A$ given in the example.  For this particular example,  one can find a rational function $\alpha_\k$ in $\k$ for which it holds that $\VV_{>0}(F_\k)=\alpha_\k\circ\T_{A}^{>0}$. However, this is not required in our definition of $F$ being $\T_A$-toric, and such a choice of $\alpha_\k$ might not exist (cf. \Cref{ex:trianglenetwork_revisited}). 
\end{example}

\begin{example}
\label{ex:square_network}
The vertical system defined by the polynomial
\[ F =-2\kappa_{1}x_{1}^9 - \kappa_{2}x_{1}^3x_{2}^4 + 2\kappa_{3}x_{2}^6 + 2\kappa_{4}x_{1}^6x_{2}^2\,\]
is  $\T_A$-invariant over $\RR_{>0}$ for 
$A=\begin{bmatrix}2 & 3\end{bmatrix}$. {It is not $\T_A$-toric, as 
 the number of $\T_A^{>0}$-cosets of $\VV_{>0}(F_\k)$ can be larger than one}: there are three for $\k=(0.01,3,1,1)$ and one for $\k=(0.01,1,1,1)$; see \Cref{fig:toric_components2}. {It is, however, locally $\T_A$-toric, as it is clear from the defining polynomial that $\VV_{>0}(F_\k)$ is 1-dimensional for all $\k\in\RR_{>0}^4$, meaning that the number of cosets is always finite.}
\end{example}

\begin{example}
\label{ex:infinitely_many_cosets}
Consider the following vertical system with $s=1$:
\[F=-\kappa_{1} x_{1}^{2} x_{3}+\kappa_{2} x_{1} x_{2}^{2}+\kappa_{3} x_{2}^{2} x_{3}\quad \in\RR[\kappa_1,\kappa_2,\kappa_3,x_1^\pm,x_2^\pm,x_3^\pm]\,.\]
Since $F_\k$ is homogeneous in $x$ for all $\k\in\RR^3_{>0}$, it is clear that $\VV_{>0}(F_\k)$ is a ruled surface, which in our language implies that $F$ is $\T_A$-invariant over $\RR_{>0}$ for $A=\begin{bmatrix}1&1&1\end{bmatrix}$ and that $\#(\VV_{>0}(F_\k)/\T_A^{>0})=\infty$.  
Hence $F$ is not locally $\T_A$-toric.
\end{example}

In the light of \Cref{def:toricity_vertical}, we take the following two-step approach to {decide whether a vertical  system $F$ is toric over $\GG$}: 
\begin{enumerate}
    \item Find a maximal-rank matrix $A$ such that $F$ is $\T_A$-invariant over $\GG$, i.e., satisfies \eqref{eq:toric_explained}. Under mild assumptions, this comes down to a linear algebra condition, which we discuss in \Cref{sec:invariance}.
    \item Bound $\#(\VV_{\GG}(F_\k)/\T_A^{\GG})$ by using concepts from 
    polyhedral and real algebraic geometry; 
       this is the topic of \Cref{sec:local,sec:counting}.
\end{enumerate}

\medskip
A key observation is that the study of  invariance can   be reduced to the case $\GG=\CC^*$.

\begin{theorem}
\label{thm:invariance_lifts_to_closure}
Let $\GG\in \{ \RR_{>0},\RR^*,\CC^*\}$  and consider a vertical system $F\in \kk[\k,x^\pm]^s$ as in \eqref{eq:NkB}. Let $\K\subseteq\GG^m$ be a Euclidean open set such that $\VV_{\GG}(F_\k)\neq\varnothing$ for some $\k\in\K$. For $A\in\ZZ^{d\times n}$, the following are equivalent:
\begin{enumerate}[label=(\roman*)]
\setlength{\itemsep}{0.2em}
    \item $\VV_{\GG}(F_\k)$ is $\T_A$-invariant over $\GG$ for all $\k\in \K$. 
    \item   $F$ is $\T_A$-invariant over $\GG$. 
    \item $F$ is $\T_A$-invariant (over $\CC^*$). 
\end{enumerate}
\end{theorem}
\begin{proof}
The implications (iii) $\Rightarrow$ (ii) $\Rightarrow$ (i) are clear. For the implication (i) $\Rightarrow$ (iii), form the complex incidence variety
\[\I=\{(\k,x)\in(\CC^*)^m\times(\CC^*)^n:F_\k(x)=0\}\,, \]
{which by \cite[Thm.~3.1]{FeliuHenrikssonPascual2025vertical} is irreducible and smooth.}
Taking Zariski closures in the complex torus, we have 
$\overline{\GG^d}=(\CC^*)^d$ and $\overline{\I\cap(\K\times\GG^n)}=\I$; this is obvious in the complex case, whereas in the real case, {we use that
a Euclidean open subset of the real part of  an irreducible smooth complex variety defined by polynomials with real coefficients is dense} (see \cite[Thm.~6.5]{Pascualescudero2022local} and \cite[Prop.~3.3.6]{BCR}). 
With this in place, for 
the map
\[\Phi\from (\CC^*)^d\times\I\to(\CC^*)^m\times(\CC^*)^n\,,\quad ( t,(\k,x))\mapsto(\k,  x\circ t^A )\,.\]
we have, using (i), that
\begin{align*}
\Phi((\CC^*)^d\times\I)&=\Phi\left(\overline{\GG^d\times (\I\cap(\K\times\GG^n))}\right)
\subseteq \overline{\Phi\left(\overline{\GG^d\times (\I\cap(\K\times\GG^n))}\right)} \\
&=\overline{\Phi(\GG^d\times (\I\cap(\K\times\GG^n)))}
\subseteq\overline{\I\cap(\K\times\GG^n)} 
\subseteq \I\,,
\end{align*}
which is equivalent to (iii) by \eqref{eq:toric_explained}. 
\end{proof}

For common parameter sets $\K$, checking that $\VV_{\GG}(F_\k)\neq\varnothing$ for some $\k\in\K$ boils down to a simple computation, as the next lemma indicates. 

\begin{lemma}
\label{lem:positive_kernel_nonempty}
Let $\GG\in \{ \RR_{>0},\RR^*,\CC^*\}$ and  $F$ be a vertical system defined by matrices $M\in\ZZ^{n\times m}$ and $C\in\kk^{s\times m}$ of rank $s$. {Let $\K\subseteq \GG^m$. }
We have that (i) $\Rightarrow$ (ii) for the following statements:
\begin{enumerate}[label=(\roman*)]
    \item $\VV_{\GG}(F_\k)\neq\varnothing$ for some $\k\in\K$.
    \item $\ker(C)\cap\GG^m\neq\varnothing$.
\end{enumerate}
Furthermore, if $\K$ is such that 
\begin{equation}\label{eq:enough_parameters}
\ker(C)\cap\GG^m\subseteq\{\k\circ x^M:\k\in\K,x\in\GG^n\}\,,
\end{equation}
then it also holds that (ii) $\Rightarrow$ (i).
In particular, (i) $\Leftrightarrow$ (ii) when $\K=\GG^m$.
\end{lemma}
\begin{proof}
(i)$\Rightarrow$(ii): If $x\in\VV_{\GG}(F_\k)$ for  $\k\in\K$, then $\k\circ x^M \in\ker(C)\cap\GG^m$. 
(ii)$\Rightarrow$(i): Let $v\in\ker(C)\cap\GG^m$. By \eqref{eq:enough_parameters}, there exists $\k\in\K$ and $x\in\GG^n$, such that $\k\circ x^M=v$, which implies $x\in \VV_{\GG}(F_\k)$.
Finally if $\K=\GG^m$, for any $v\in \ker(C)\cap\GG^m$, 
we have $v=v\circ \1^M$, with $\1$ the vector of all ones.
\end{proof}

\begin{remark}\label{rmk:equivalent_descriptions_of_cosets}
  {Over $\GG=\RR_{>0}$}, it is a well-known fact (see, e.g., \cite{Feinberg1995class}, \cite[§3.2]{Muller2015injectivity}) that
\[\alpha\circ\T_{A}^{>0}=\{x \in\RR^n_{>0}:\log(x)-\log(\alpha)\in\ker(A)\}=\{x\in\RR^n_{>0}:x^B=\alpha^B\}\,,\]
where $B\in\ZZ^{n\times(n-\rk(A))}$ is a matrix whose columns form a basis for $\ker_{\QQ}(A)$. Cosets of this form are  sometimes called \emph{log-parametrized} sets in the context of reaction networks \cite{Hernandez2023independent} and \emph{log-linear} or \emph{log-affine} sets in the context of algebraic statistics, see, e.g., \cite[Ch.~6--7]{Sullivant2018}.
\end{remark}

\begin{remark}
\label{rmk:binomial_implies_toric}
When  $\GG=\RR_{>0}$, any ideal $I\subseteq\RR[x]$ generated by binomials $ax^u-bx^v$ for $a,b\in\RR_{>0}$ and $u,v\in\ZZ_{\geq 0}^n$ gives a toric zero locus $\VV_{>0}(I)$ (see, e.g., \cite[Prop.~5.2]{conradi2019total}). 
When $M\in\ZZ_{\geq 0}^{n\times m}$, a computational sufficient condition for 
a vertical system $F$ {to be \emph{generically} toric} over $\RR_{>0}$ is therefore that a reduced Gröbner basis for the ideal $\langle F\rangle\subseteq\RR(\k)[x]$ is binomial. If the binomial Gröbner basis specializes for each $\k\in\RR^m_{>0}$, the conclusion can be strengthened to {$F$ being toric} 
\cite[Rk.~2.4]{sadeghimanesh:multi}. A particularly simple special case where we can immediately assert {that $F$ is toric} 
is when putting $C$ in reduced row  echelon form gives a matrix where all rows have supports of size 2.
\end{remark}

We conclude the section with some additional examples to illustrate some subtleties of the concepts defined above.

\begin{example}
To see that {a system can be generically toric but not toric},
 consider the linear vertical system 
\[F = \left[\begin{array}{l}
\kappa_{1} x_{1} - \k_2 x_2 \\
\kappa_{1} x_{1} - \k_3 x_2
\end{array}\right]\quad \in \CC[\kappa_1,\k_2,\k_3,x_1^\pm,x_2^\pm]^2.\]
As the zero set is generically one point, $F$ is generically $\T_A$-toric with $A$ the empty matrix. However, when $\k_2=\k_3=1$, the number of $\T_A$-cosets is infinite as $\VV_{\CC^*}(F_\k)$ is one dimensional. 
\end{example}

\begin{example}
\label{ex:different_dependencies}
The linear dependencies among the coefficients of a vertically parametrized system play an important role in whether the system {is toric}. Consider   the following two systems, with the same support but different dependencies:
\[\left[{\small\begin{array}{l}-3 \kappa_{1} x_{1}^6 + 3 \kappa_{2} x_{1}^3 x_{2}^2  + 3 \kappa_{3} x_{2}^4 - \kappa_{4} x_{1} + \kappa_{5} x_{3}^5\\
\kappa_{1} x_{1}^6 - \kappa_{2} x_{1}^3 x_{2}^2 - \kappa_{3} x_{2}^4 + \kappa_{4} x_{1} - \kappa_{5} x_{3}^5\end{array}}\right]\quad
\text{and}
\quad
\left[{\small\begin{array}{l} -\kappa_{1} x_{1}^6 + \kappa_{2} x_{1}^3 x_{2}^2 + \kappa_{3} x_{2}^4  - \kappa_{4} x_{1} + \kappa_{5} x_{3}^5\\
\kappa_{6} x_{1}^6 - \kappa_{7} x_{1}^3 x_{2}^2 - \kappa_{8} x_{2}^4  + \kappa_{9} x_{1} - \kappa_{10} x_{3}^5\end{array}}\right]\,.\]
The first system {is $\T_A$-toric} over $\RR_{>0}$ for $A=[\,10\ \ 15\ \ 2\,]$, whereas the second  is not toric.
\end{example}

\begin{example}\label{ex:reductionC}
A key ingredient in the proof of
\Cref{thm:invariance_lifts_to_closure} is that the incidence variety of a vertical system is irreducible. Failure of this condition may give rise to systems displaying invariance over $\RR_{>0}$ but not over $\CC^*$. A simple example  is the system with one entry
\[ F = (\k_1 x_1 - \k_2 x_2)(\k_3x_1^2+\k_4x_2^4)\,,\]
which is $\T_A$-invariant over $\RR_{>0}$ for $A=\begin{bmatrix}1 & 1\end{bmatrix}$, but $\VV_{\CC^*}(F_\k)$ consists of three connected components, only one of which  has $\T_A$-invariance.
\end{example}

\begin{remark}
We only focus on the zeros with nonzero or positive coordinates as the behavior on the coordinate hyperplanes might vary drastically. If such solutions are of interest, one may perform a systematic case-by-case analysis, where different combinations of variables are set to zero, and the resulting system is studied with the methods of this paper. 
\end{remark}

\section{Characterization of toric invariance}
\label{sec:invariance}

We now give a characterization of toric invariance of a vertically parametrized system $F$.
The key object is the \term{toric invariance lattice}, defined as
\begin{equation*}
\cL_F =\left\{a \in\ZZ^{n}: F \text{ is }\T_{a}\text{-invariant}\right\}\,,
\end{equation*}
which is an abelian subgroup of $\ZZ^n$. In the language of toric geometry, $\cL_F$ is the \emph{cocharacter lattice} of the maximal subtorus of $(\CC^*)^n$ that stabilizes all of the varieties $\VV_{\CC^*}(F_\k)$. In concrete terms, we obtain the unique maximal-dimensional torus $\T_A$ for which the system $F$ is invariant by taking the rows of $A$ to be any $\ZZ$-basis for $\cL_F$.

\begin{proposition}
\label{prop:characterization_of_invariance}
Let $\GG\in \{ \RR_{>0},\RR^*,\CC^*\}$, $A\in\ZZ^{d\times n}$, and $F$ be a vertical system 
such that $\Z_{\GG}\neq \varnothing$. 
Then $F$ is $\T_A$-invariant over $\GG$ if and only if $\row_{\ZZ}(A)\subseteq \cL_F$.
\end{proposition}
\begin{proof}
This is a direct consequence of \Cref{thm:invariance_lifts_to_closure}, \Cref{lem:positive_kernel_nonempty}, and the discussion above. 
\end{proof}

A {stronger} condition for a row vector $a\in\ZZ^n$ to belong to $\cL_F$ is that  every entry of $F$ is \term{quasihomogeneous} with weights given by $a$, in the sense that there exists $b\in\ZZ^{s}$ such that
\[{F_\k(x\circ t^a)=  F_\k(x)  \circ t^b \,\quad \text{holds in } \CC[ t^\pm,x^\pm]\, \text{ for all $\k\in (\CC^*)^m$}.}\]

\begin{lemma}\label{lem:quasihomogeneous}
  A vertical system $F$ with defining matrices $M\in\ZZ^{n\times m}$ and $C\in \CC^{s\times m}$ of  rank $s$ is quasihomogeneous with weights given by a row vector $a\in\ZZ^n$  if and only if $a$ is perpendicular to the affine hull of the Newton polytope of each of the $s$ polynomials in $F$: 
\begin{equation}
\label{eq:condition_for_quasihomogeneity}
    a\, M_j=a\, M_{j_0} \quad \text{for all }j,j_0\in\supp(C_{i*}) \text{ and all }i\in[s]\,.
\end{equation}
\end{lemma}
\begin{proof}
  Follows from  \cite[Prop.~6.1.2(a)]{gelfand1994discriminants}.
\end{proof}
    
\Cref{lem:quasihomogeneous}  {gives a sufficient condition for invariance, which depends on the  supports of the rows of $C$. We might obtain additional sufficient conditions by performing row operations  on $C$, as these preserve $\VV_{>0}(F_\k)$ but may change the support of the rows.  The main   message of this section is that these sufficient conditions, relying on quasihomogeneity, together completely characterize $\T_a$-invariance. In practice, it is enough to consider quasihomogeneity
when $C$ is in reduced row echelon form. }

The structure of $\ker(C)$ is captured by the  column matroid $\mathcal{M}_C$ of $C$, which is the matroid on $[m]$  whose circuits are the minimal nonempty index sets of linearly dependent columns of $C$.

\begin{definition}
The \term{matroid partition} defined by a matrix $C\in\CC^{s\times m}$ is the partition on $[m]$ whose blocks are the connected components of the column matroid $\mathcal{M}_C$ of $C$. In other words, the blocks  are the equivalence classes of the equivalence relation $\sim$ on $[m]$ generated by
 \[j\sim j_0 \quad \text{if }j,j_0 \in c \quad  \text{ for some  circuit $c$ of $\mathcal{M}_C$}\, .\]
 \end{definition}

By \cite[Thm.~5.2.1]{Welsh1976} the blocks of the matroid partition are also    the connected components of the {dual matroid of $\mathcal{M}_C$}, whose circuits are the minimal supports of nonzero vectors in $\row(C)$. 
The following easy lemma tells us that we do not need to find all circuits  to determine the matroid partition: a basis consisting of vectors with minimal support, for example given by the rows of $C$ if it is in  reduced row echelon form, suffices.
We include the proof for completeness.

Note that  a nonzero vector $v\in V$ in a vector space $V$ having
\emph{minimal support} means that the zero vector is the only vector of $V$ with support strictly contained in $\supp(v)$.

\begin{lemma}
\label{lem:partition_computable_from_generating_set}
Let $C\in\CC^{s\times m}$  and $\mathcal{V}$ be 
  a basis of $\ker(C)$ or of $\row(C)$ consisting of vectors of minimal support. 
The blocks of the matroid partition are the equivalence classes of the equivalence relation $\sim$  on $[m]$ generated by
\begin{equation}\label{eq:partition_small}
    j\sim j_0 \quad \text{if }j,j_0 \in \supp(v) \quad \text{ for some  $v\in \mathcal{V}$}\, .
 \end{equation}
\end{lemma}
\begin{proof}
Let $V=\row(C)$ or $V=\ker(C)$ such that $\mathcal{V}=\{v_1,\dots,v_\ell\}$ is a basis of $V$. 
For a nonzero vector  $w\in V$ of minimal support, let $\rho$ be a block
of the  partition defined  by \eqref{eq:partition_small} such that $\supp(w)\cap \rho\neq \varnothing$. It is enough to show that $\supp(w)\subseteq \rho$. 
As  $\mathcal{V}$ is a basis of $V$, we can write 
\[ w = \sum_{\supp(v_j) \subseteq \rho} \lambda_j v_j+ \sum_{\supp(v_j) \cap  \rho =\varnothing} \lambda_j v_j\,. \]
The two summands have disjoint index sets and define vectors in $V$. 
As $w$ has minimal support and the first summand is nonzero,  the second summand must vanish. 
In particular, 
$\supp(w)\subseteq \rho$ as desired.
\end{proof}

\begin{remark}
\label{rmk:other_partitions}
The matroid partition has previously appeared in the reaction networks literature under the term \emph{the finest independent decomposition} \cite[Thm.~3.3]{hernandez2021}, referring to Feinberg's notion of  \emph{independent decompositions} of a network \cite{FEINBERG19872229}.  
\end{remark}

\begin{remark}\label{rem:intersection}
If 
we write the matroid partition as $[m]=\rho_1\sqcup \dots \sqcup \rho_{\theta}$ and for $v\in \ker(C)$, we let $v^{(i)}\in \CC^m$, for all $i\in [\theta]$, be the unique vectors with $v=v^{(1)}+\dots + v^{(\theta)}$ and $\supp(v^{(i)}) \subseteq \rho_i$, then  $v^{(i)} \in \ker(C)$ as well for all $i$. In other words, 
$\ker(C)$ is, after a suitable coordinate permutation,  the direct {product} of the kernels of the matrices $C^{(i)}$ obtained from $C$ by considering the columns  with indices in $\rho_i$.
In particular    
\[{\VV_{\CC^*}(F_\k)= \VV_{\CC^*}(F_\k^{(1)}) \cap \dots \cap \VV_{\CC^*}(F_\k^{(\theta)})\,}\]
where the vertical system $F^{(i)}$ is defined by considering $C^{(i)}$ and the exponent matrix $M^{(i)}$ defined analogously.
 \end{remark}

\begin{theorem}
\label{thm:characterization_of_invariance}
Let $F$ be a vertical system with defining matrices $M\in\ZZ^{n\times m}$ and $C\in \CC^{s\times m}$ of  rank $s$ with $\ker(C)\cap (\CC^*)^m\neq\varnothing$. 
Let  $[m]=\rho_1 \sqcup \dots \sqcup \rho_\theta$ be its  matroid partition and $\cL_F$ the toric invariance lattice.
Then $a\in\cL_F$ if and only if $a M_{j}=a M_{j_0}$ for all $j,j_0\in \rho_i$ and all $i=1,\dots,\theta$.
Equivalently, it holds that
\[\cL_F=\bigcap_{i=1}^\theta \spn_{\ZZ}\{M_j-M_{j_0}:j,j_0\in\rho_i\}^\perp\, . \]
\end{theorem}
\begin{proof}
The reverse implication follows from the facts that  replacing $C$ by a matrix $C'$ in reduced row echelon form with the same row span as $C$ in $F$ does not change $\mathcal{L}_F$, and  that
\Cref{lem:quasihomogeneous,lem:partition_computable_from_generating_set} imply that $C'(\k\circ x^M)$ is quasihomogeneous with weights given by $a$. 

 \smallskip
For the forward implication, by \eqref{eq:toric_explained}, $a\in\cL_F$ if and only if 
$(\k\circ x^M)\circ t^{aM} \in \ker(C)$ for all $t\in \CC^*$ if $\k\circ x^M\in \ker(C)$.
As \eqref{eq:enough_parameters} holds  with equality for $\K=(\CC^*)^m$ and $\GG=\CC^*$, 
and using that $\ker(C)\cap (\CC^*)^m$ is Zariski dense in $\ker(C)$, it holds that 
\[   C(v\circ t^{aM})=0 \quad \text{ for all } v  \in \ker(C)  \text{ and } \ t\in \CC^*\,.\]
This can only happen if all coefficients of $C(v  \circ t^{aM})$ as a polynomial in $t$ are zero. 
Thus, for any nonzero $v\in \ker(C)$ and $j_0\in\supp(v)$, the coefficient of $t^{aM_{j_0}}$ must be zero:
\[\sum_{j\in\supp(v), \ aM_j=aM_{j_0}} C_j v_j=0\,.\]
We obtain a nonzero   vector $w\in\ker(C)$ by setting
$w_j=v_j$ if $aM_j=aM_{j_0}$ and $0$ otherwise. Taking  $v$ with minimal support,
we conclude that $w=v$ and hence 
$aM_j=aM_{j_0}$ for all $j\in\supp(v)$. The conclusion now follows from the definition of the matroid partition. 

The last statement is  immediate. 
\end{proof}

\begin{remark}
{If the support of each row of $C$ is contained in one block of the matroid partition,} then, by \Cref{thm:characterization_of_invariance}, we have quasihomogeneity for all weights in the toric invariance lattice. This is for example the case when the matroid partition is trivial or when the rows of the coefficient matrix $C$ have pairwise disjoint support, as is the case for freely parametrized systems. 
\end{remark}
 
\begin{remark}\label{rk:Cayley}
To find a matrix $A$ of maximal rank for which a vertical system $F=C(\k\circ x^M)$ is $\T_A$-invariant, we consider the Cayley matrix $\widehat{M}\in \ZZ^{(n+\theta)\times m}$, obtained by appending to $M$, for each block $\rho$ of the matroid partition, the row  with $1$'s for the indices in $\rho$ and zero otherwise. Let $\widehat{A}$ be a matrix whose rows form a basis for $\ker_{\ZZ}(\widehat{M}^\top)$, and let $A$ be the first $n$ columns of $\widehat{A}$. Then the rows of $A$ form a basis for $\cL_F$. 
\end{remark}

\Cref{thm:characterization_of_invariance}, \Cref{prop:characterization_of_invariance}
and \Cref{lem:positive_kernel_nonempty} lead to \Cref{alg:invariance} for finding a maximal rank matrix $A$ for which $F$ displays toric invariance over $\GG\in\{\RR_{>0}, \RR^*, \CC^*\}$. 

\begin{algorithm}[H]
\caption{Toric invariance}
\SetAlgoNoLine
\fontsize{10}{13}\selectfont
\KwIn{Matrix $C\in \kk^{s\times m}$ of full rank $s$, matrix $M\in\mathbb{Z}^{n\times m}$.}
\KwOut{Rank $d$ matrix $A\in \mathbb{Z}^{d\times n}$ such that $F=C(\k\circ x^M)$ is $\T_A$-invariant over $\mathbb{G}$.}

Find a basis $\{E_1,\dots,E_{m-s}\}$ for $\ker(C)$ {consisting of vectors with minimal support}

\If{(Union of supports of $E_1,\dots,E_{m-s}$) $\neq [m]$}{
    \Return{$\mathbb{V}_{\mathbb{G}}(F_\k)=\varnothing$ for all $\k\in \mathbb{G}^m$}
}

\If{$\mathbb{G}=\mathbb{R}_{>0}$ and $\text{(Interior of the polyhedral cone $\ker(C)\cap\mathbb{R}_{\ge 0}^m$)}=\varnothing$}{
    \Return{$\mathbb{V}_{>0}(F_\k)=\varnothing$ for all $\k\in \mathbb{R}_{>0}^m$}
}

Find the matroid partition $\{\rho_1,\dots,\rho_{\theta}\}$ from the supports of $\{E_1,\dots,E_{m-s}\}$

Construct $\widehat{M}\in \mathbb{Z}^{(n+\theta)\times m}$ from \Cref{rk:Cayley}

\eIf{$\mathbb{G}=\mathbb{R}_{>0}$ or $\mathbb{C}^*$}{
    Find a $\mathbb{Z}$-matrix $\widehat{A}$ with rows a basis for $\ker_{\mathbb{Q}}(\widehat{M}^\top)$
}{
    Find a $\mathbb{Z}$-matrix $\widehat{A}$ with rows a basis for $\ker_{\mathbb{Z}}(\widehat{M}^\top)$
}

\Return{$A=\text{(first $n$ columns of $\widehat{A}$)}$}
\label{alg:invariance}
\end{algorithm}

\begin{remark}\label{rem:true_dim}
\Cref{thm:characterization_of_invariance} implies  that  invariance simply encodes that the system depends, in a sense, on fewer than $n$ variables. 
The key insight is that by \Cref{thm:characterization_of_invariance}, the affine span of the columns of the matrices $M^{(i)}$ defined in \Cref{rem:intersection} is contained in a parallel translate of $\row(A)^\perp$. 
By assuming that $C$ is in reduced row echelon form, choosing an index $j_i\in \rho_i$ for each $i$, and multiplying  by $x^{-M_{j_i}}$ the entries of $F$ that correspond to a row of $C$ with support in $\rho_i$, 
we obtain an equivalent system whose exponents affinely span $\row(A)^\perp$. 

By a monomial change of coordinates, this system can be transformed to a  system depending on $n-\rk(A)$ variables and with $\rk(A)$ free variables. 
Hence, in the new variables, the zero set is a union of translates of a linear subspace of dimension $\rk(A)$, and there is one set per zero of the system in $n-\rk(A)$ variables. 
Reversing the change of coordinates, the linear subspace becomes $\T_A$. 
Observe that invariance does not imply that the affine span of the columns of $M$ is not full dimensional, unless the matroid partition is trivial. 
\end{remark}

\begin{example}\label{ex:true_dim}
\Cref{rem:true_dim} does not necessarily hold for systems that are not vertical, in the sense that invariance may arise due to mechanisms other than the system depending on fewer variables in disguise.  To see this, consider the system 
\[ F = \left[ \begin{array}{l}  \k_1 x_1x_3 - \k_1 x_2x_3 \\[4pt] \k_2 x_1 - \k_2 x_2 + \k_3 x_1^2 - \k_4 x_3^2  \end{array}\right] \,.\]
By simple inspection, the zeros over $\CC^*$ satisfy $x_1=x_2=\pm x_3$, and hence the system is $\T_A$-invariant for $A=[1 \ 1 \ 1]$. However, there is no linear combination of the entries of $F$, nor translation of the exponents, that equivalently transform the system into one with exponents in a lower dimensional affine subspace. 
\end{example}

\begin{example}\label{ex:fig}
Consider the first vertical system of \Cref{ex:different_dependencies}, 
given by the matrices
\[C=\begin{bmatrix}-3 & 3  & 3 & -1 & 1\\ 1 & -1  & -1 & 1 & -1\end{bmatrix}\quad\text{and}\quad 
M=\begin{bmatrix} 6 & 3 & 0 & 1 & 0\\ 0 & 2 & 4  & 0 & 0 \\ 0 & 0  & 0 & 0 & 5\end{bmatrix}\,.\]
Gaussian elimination of $C$ gives
the matroid partition $[m]=\{1,2,3\}\sqcup\{4,5\}$. 
 The left-kernel of $\widehat{M}$ has dimension $1$.  We obtain that
 $F$ is $\T_{a}$-invariant for $a=(10,15, 2)$ and there is no invariance for a higher dimensional torus. 
In this example,
there is 
only trivial quasihomogeneity with all weights zero. This implies that the second system in \Cref{ex:different_dependencies} {is not toric}. 
\end{example}

\begin{example}
\label{ex:IDH_invariance}
In \Cref{ex:IDH},  Gaussian elimination of $C$
gives that the matroid partition is trivial and toric invariance agrees with quasihomogeneity. We find that 
$\cL_F=\spn_{\ZZ}\{(1,1,-1,0,0), (0,0,0,1,0),(0,0,1,0,-1)\}^\perp\,$
and a $\ZZ$-basis is given by the rows of
\[A=\begin{bmatrix}1 & 0 & 1 & 0 & 1\\
0 & 1 & 1 & 0 & 1\end{bmatrix}.\]
We conclude that the $2$-dimensional torus $\T_A^{>0}$ is the maximal-dimensional torus for which $F$ is invariant over $\RR_{>0}$. We recover the same matrix $A$ given in  \Cref{ex:IDH}. 
\end{example}

\begin{example}\label{ex:emptygenerically}
    Toric invariance for a nonempty matrix $A$ does not imply that the varieties $\VV_{\CC^*}(F_\k)$ are generically nonempty. For the vertical system
    \[F = \left[ \begin{array}{l} \k_1 x_1^2 - \k_2 x_1x_2 \\[4pt]
\k_3 x_1^2 - \k_2 x_1x_2
\end{array}\right]\, , 
\]
the toric invariance lattice is 
$\cL_F=\spn_{\ZZ}\{(1,1)\}$, giving that $\VV_{\CC^*}(F_\k)$ is 
$\T_{[1 \ 1]}$-invariant. However, $\VV_{\CC^*}(F_\k)\neq \varnothing$ only if $\k_1= \k_3$. 
\end{example}

\begin{remark}
Recall that the lattice $\cL_F$ describes the largest-dimensional torus for which  $\VV_{\CC^*}(F_\k)$ is invariant for \textit{all} $\k\in (\CC^*)^m$, and by \Cref{thm:invariance_lifts_to_closure}, it also describes the largest-dimensional $\GG$-torus  for which  $\VV_{\GG}(F_\k)$ is invariant for all $\k\in \K$, as long as $\K\subseteq \GG^m$ is   Euclidean open and $\VV_{\GG}(F_\k)\neq \varnothing$  for some $\k\in \K$. 
However, a smaller subfamily might   display invariance under a larger-dimensional torus if the parameter set  is   Euclidean closed.
For instance, with $\GG=\RR_{>0}$, the vertical system $F=(\kappa_1-\kappa_2)x_1 x_2+\kappa_3x_2^2-\kappa_4 x_1^3$ satisfies
 $\cL_F=\{0\}$. 
However, for  
$\K=\VV(\kappa_1-\kappa_2)\cap\RR_{>0}^4$, it holds that $\VV_{>0}(F_\k)$ is invariant for $A=\begin{bmatrix} 2 & 3 \end{bmatrix}$.
\end{remark}

\begin{remark}
In \cite[Thm.~5]{Muller:inequalities}, $\VV_{>0}(F_\k)$ is expressed as the union of toric cosets (denoted $Z_c$ in loc. cit.).  Adapted to our setting, the start point of that work is a decomposition of $\ker(C) \cap \RR^m_{>0}$ as a direct product of cones. By choosing the decomposition given by the matroid partition, their expression agrees, as expected, with ours. An analogous construction is given in \cite{Banaji:Feliu}, building on \cite[§2.4]{banaji:bifurcations}.
\end{remark}

\section{{Locally toric systems}}
\label{sec:local}

Having invariance for vertical systems determined, in this section, we give a complete characterization of {systems that are generically locally toric}, as well as sufficient conditions for  
{a system to be locally toric}.
A useful observation is that if $F$ is $\T_A$-invariant, then
\begin{equation}
\label{eq:dimension_bound_invariance}
\dim(\VV_{\CC^*}(F_\k))\geq\dim(\T_A)=\rk(A)\,,
\end{equation}
for all $\k\in \Z_{\CC^*}$. If in addition $\VV_{\CC^*}(F_\k)$ is locally $\T_A$-toric, then $\dim(\VV_{\CC^*}(F_\k))=\rk(A)$.

\begin{lemma}\label{lem:sum_dim}
Let $F$ be a vertical system with defining matrices $M\in\ZZ^{n\times m}$ and $C\in \CC^{s\times m}$ of  rank $s$. 
Assume that $F$ is $\T_A$-invariant  for a matrix $A\in\ZZ^{d\times n}$ of rank $d$.
If $F$ is nondegenerate over $\CC^*$, 
then $s+d\leq n$. 
\end{lemma}

\begin{proof}
This follows from \eqref{eq:dimension_bound_invariance},
as $\dim(\VV_{\CC^*}(F_\k))=n-s$ for some $\k$ by \Cref{prop:vertical}.
\end{proof}

The inequality in \Cref{lem:sum_dim} might not hold when $F$ is degenerate, as \Cref{ex:emptygenerically} shows. The following   lemma will be at the core of several results below. 

\begin{lemma}\label{lem:nonsingular}
Let $\GG\in \{ \RR_{>0},\RR^*,\CC^*\}$, and  $F$  be a vertical system defined by $M\in\ZZ^{n\times m}$ and $C\in \CC^{s\times m}$ of  rank $s$,  which is $\T_A$-invariant  for a   matrix $A\in\ZZ^{d\times n}$ of rank $d$. 
Assume that for a fixed $\k\in \GG^m$, $\VV_{\GG}(F_\k)\neq \varnothing$ and all irreducible components of $\VV_{\CC^*}(F_\k)$ that intersect $\GG^n$ contain a nondegenerate zero of $F_\k$ in $\GG^n$. Then $\VV_{\GG}(F_\k)$ is locally $\T_A$-toric over $\GG$  if and only if $n=s+d$.  
\end{lemma}
\begin{proof}
As $F$ is $\T_A$-invariant, it is also $\T_A$-invariant  over $\GG$ by \Cref{thm:invariance_lifts_to_closure}. 
Let $\bigcup_{i=1}^\ell Y_i$ be the union of the irreducible components of 
$\VV_{\CC^*}(F_\k)$ that intersect $\GG^n$. The condition on nondegeneracy guarantees that  $\dim(Y_i)=n-s$  for all $i$   and that 
$\overline{\VV_{\GG}(F_\k)} = \bigcup_{i=1}^\ell Y_i$, where the overline denotes the Zariski closure in $(\CC^*)^n$. 
 For a coset $\alpha\circ \T_A^\GG$ with $\alpha\in \VV_{\GG}(F_\k)$,   
we have  $\alpha\circ\T_A$ must be contained in $Y_i$ for some $i\in[\ell]$ by irreducibility. 
As $\dim(\alpha\circ \T_A)=d$, if $d=n-s$, then $\alpha\circ \T_A=Y_i$ and {the system is locally toric}.   Conversely,  if 
$\VV_{\GG}(F_\k)$ is locally $\T_A$-toric over $\GG$, then  $\VV_{\GG}(F_\k) = \bigsqcup_{i=1}^p \alpha_{i} \circ \T_A^\GG$ for some $\alpha_i\in \VV_{\GG}(F_\k)$, $i\in [p]$. Hence, 
\[n-s=\dim(\bigcup\nolimits_{i=1}^\ell Y_i)=\dim\left(\overline{\bigsqcup\nolimits_{i=1}^p \alpha_{i}\circ\T_A^{\GG}}\right)=
\dim\left( \bigsqcup\nolimits_{i=1}^p \alpha_{i} \circ\T_A \right) = d\,.\qedhere
\] 
\end{proof}

\begin{theorem}
\label{thm:characterization_generic_local_toricity}
Let $\GG\in \{ \RR_{>0},\RR^*,\CC^*\}$ and $F$ be a vertical system with defining matrices $M\in\ZZ^{n\times m}$ and $C\in \CC^{s\times m}$ of  rank $s$ with $\ker(C)\cap \GG^m \neq \varnothing$.
Assume that $F$ is $\T_A$-invariant   for a matrix  $A\in\ZZ^{d\times n}$ of rank $d$ such that $s+d\leq n$.
The following statements are equivalent:
\begin{itemize}
\item[(i)] $F$ is  nondegenerate  over $\CC^*$ and $n=s+d$.
\item[(ii)] $F$ is generically   locally $\T_A$-toric over $\CC^*$.
\item[(iii)] $F$ is generically   locally $\T_A$-toric over $\GG$.
\item[(iv)] $F$ is  locally $\T_A$-toric over $\GG$ for parameters in a nonempty Euclidean open set of $\GG^m$.
\end{itemize}
\end{theorem}

\begin{proof}
(i) $\Rightarrow$ (ii): 
By  \Cref{prop:vertical}(i), for all $\k$ in a nonempty Zariski open subset  of $\Z_{\CC^*}$, all zeros of $F_\k$ in $(\CC^*)^n$ are nondegenerate.
Hence (ii) follows from \Cref{lem:nonsingular}.

\smallskip
(ii) $\Rightarrow$ (iii): 
As $\ker(C)\cap \GG^m \neq \varnothing$, \Cref{prop:vertical} gives that  $F$  is nondegenerate over $\GG$ and $\Z_{\GG}$ is Zariski dense in $\CC^m$. Since  $\VV_{\CC^*}(F_\k)$ is locally $\T_A$-toric for all  $\k$ in a  nonempty Zariski open set  $\mathcal{U}\subseteq \Z_{\CC^*}$, it follows that $\VV_{\GG}(F_\k)$  is locally $\T_A$-toric over $\GG$ for all $\k$ in the 
nonempty Zariski open set  $\mathcal{U}\cap \Z_{\GG}$ of $\Z_{\GG}$ as desired. 
 
\smallskip
(iii) $\Rightarrow$ (iv): This is clear, as $F$ is locally $\T_A$-toric for all $\k$ in a nonempty Zariski open set of $\Z_\GG$, which has nonempty Euclidean interior as $F$ is nondegenerate. 

\smallskip
(iv) $\Rightarrow$ (i): Follows from \Cref{lem:nonsingular} together with the fact that $\Z_\GG$ has nonempty Euclidean interior by (iv) and hence $F$ is nondegenerate by \Cref{prop:vertical}(i). 
\end{proof}

\begin{example}
The vertical system $F$ in \Cref{ex:infinitely_many_cosets} has invariance for the maximal-rank matrix 
$A=\begin{bmatrix}1 & 1 & 1\end{bmatrix}$. As $s+d=2<3=n$, $F$ is not generically locally $\T_A$-toric. As $F$ is nondegenerate,  $\VV_{>0}(F_\k)$ is a union of infinitely many lines whenever not empty. 
\end{example}
 
\Cref{thm:characterization_generic_local_toricity}  completely characterizes {when a system is generically locally toric}, once a maximal-rank matrix $A$ for which $F$ is invariant has been found. {The equivalence of (ii) and (iii) shows that being generically locally toric does not depend on $\GG$}, as long as $\ker(C)\cap \GG^m\neq \varnothing$. 
{In \Cref{prop:nondeg} below, we give a sufficient condition for being locally toric in all of $\GG^m$ (not just generically), where the choice of $\GG$ plays a bigger role.}

\begin{proposition}\label{prop:nondeg}
Let $\GG\in \{ \RR_{>0},\RR^*,\CC^*\}$ and $F$ be a vertical system with defining matrices $M\in\ZZ^{n\times m}$ and $C\in \kk^{s\times m}$ of  rank $s$ with $\ker(C)\cap \GG^m \neq \varnothing$. 
Assume that $F$ is $\T_A$-invariant   for a matrix $A\in\ZZ^{d\times n}$ of rank $d$ with $n=s+d$.
If   
\begin{equation}\label{eq:nondeg_condition_all}
\rk(C \diag(w) M^\top)=s \quad \text{for all }w\in \ker(C)\cap \GG^m, 
\end{equation}
then $F$ is locally $\T_A$-toric over $\GG$ and 
$\dim(\VV_{\GG}(F_{\k}))=n-s$ for all $\k\in \GG^m$.
\end{proposition}
\begin{proof}
By  \Cref{prop:vertical}(i), 
all zeros of $F_{\k}$ in $\GG^n$ are nondegenerate for all $\k\in \GG^m$, so the statement follows from \Cref{lem:nonsingular}. 
\end{proof}

It might seem that condition \eqref{eq:nondeg_condition_all} in \Cref{prop:nondeg} is very strict, but 
in practice, it applies to many realistic reaction networks, as we will see in \Cref{sec:crnt_perspectives}.
In that application, we have $\GG=\RR_{>0}$ and   \eqref{eq:nondeg_condition_all} 
can be checked   using the parametrization of $\ker(C)\cap \RR^m_{>0}$ given by the generators of the polyhedral cone $\ker(C)\cap \RR^m_{\geq 0}$, as the next example illustrates.

\begin{example}\label{ex:idhkp2}
For \Cref{ex:IDH} (with $n=5$ and $s=3$), we have
\[ \ker(C)\cap \RR^6_{>0} = \{ \lambda_1(1,0,1,1,0,1)  + \lambda_2(0,0,0,1,1,0)  + \lambda_3(1,1,0,0,0,0) : \lambda\in \RR^3_{>0} \}.\]
From this parametrization we obtain for $w\in \ker(C)\cap \RR^6_{>0}$, 
\[C \diag(w) M^\top =\left[\begin{smallmatrix}\lambda_{{1}}+\lambda_{{3}}&\lambda_{{1}
}+\lambda_{{3}}&-\lambda_{{3}}&0&-\lambda_{{1}}\\ 0&0
&\lambda_{{1}}&0&-\lambda_{{1}}\\ 0&0&\lambda_{{1}}+
\lambda_{{2}}&\lambda_{{1}}+\lambda_{{2}}&-\lambda_{{1}}-\lambda_{{2}}\end{smallmatrix}\right],
\]
which for instance has the $3\times 3$ minor $(\lambda_1+\lambda_3)\lambda_1(\lambda_1+\lambda_2)$ given by columns $1,3,4$. 
Therefore, \eqref{eq:nondeg_condition_all} holds and 
$F$ is locally $\T_A$-toric with $\dim(\VV_{>0}(F_{\k})) =2$ for all $\k\in \RR^6_{>0}$.
{For $\GG=\RR^*$, condition \eqref{eq:nondeg_condition_all}
fails for $\lambda_1=-\lambda_2$,  and hence \Cref{prop:nondeg} is not informative. However, we still have that the system is \emph{generically} locally toric over $\RR^*$ (cf. \Cref{thm:characterization_generic_local_toricity}).} 
\end{example}

The results of this section yield a procedure to detect {whether a system is (generically) locally toric}, when \Cref{alg:invariance} returns a matrix $A\in \ZZ^{d\times n}$. Namely if $s+d<n$, then we readily conclude that $F$ is not generically locally $\T_A$-toric over $\GG$. Otherwise, all we need is to check nondegeneracy, which can be verified by computing $r:=\rk(C \diag(w) M^\top)$ for a randomly generated $w\in \ker(C)$ and deciding whether $r=s$. If this is not the case, then one should verify symbolically that the rank $r$ is smaller for all $w$. Finally,  {$F$ can be certified to be locally toric} if \eqref{eq:nondeg_condition_all} holds.
These steps are incorporated in \Cref{alg:summary} below, given for $\GG=\RR_{>0}$. 

\section{Counting the number of cosets in $\RR^n_{>0}$}
\label{sec:counting}

When $F$ is locally $\T_A$-toric,  
the next question is to decide how many cosets there are. When $\GG=\RR^*$ or $\CC^*$, the number of cosets can sometimes be found by counting the number of intersection points between $\VV_{\GG}(F_\k)$ and a certain toric variety, as discussed in \cite{Hubert2012rational} for quasihomogeneity.
{We focus now on the case $\GG=\RR_{>0}$, where it turns out that the number of cosets is always the number of intersections with a linear variety.}

\subsection{The coset counting system}
Consider a vertical system $F$ with defining matrices $M\in\ZZ^{n\times m}$ and $C\in \RR^{s\times m}$ of  rank $s$,  satisfying $\ker(C)\cap \RR^m_{>0} \neq \varnothing$. 
Assume that $F$ is $\T_A$-invariant for a  matrix $A\in\ZZ^{d\times n}$ with $n=s+d$, and consider the augmented vertical system 
\begin{equation}
\label{eq:system_number_of_cosets}
H=\Big( C(\k\circ x^M),Ax-b \Big) \in \RR[\k,b,x^\pm]^{n}.
\end{equation}
{As the following proposition makes precise, the points in $\VV_{>0}(H_{\k,b})$ count the number of $\T_A$-cosets of $\VV_{>0}(F_\k)$, and we will therefore refer to  \eqref{eq:system_number_of_cosets} as the \term{coset counting system} of $F$.}

\begin{proposition}
\label{prop:coset_counting}
Let $F$ be a vertical system with defining matrices $M\in\ZZ^{n\times m}$ and $C\in \RR^{s\times m}$ of  rank $s$ with $\ker(C)\cap \RR^m_{>0} \neq \varnothing$. 
Assume that $F$ is $\T_A$-invariant for a  matrix $A\in\ZZ^{d\times n}$ and let $H$ be the coset counting system \eqref{eq:system_number_of_cosets}.
Then, for a given $\k\in \RR^m_{>0}$ and any $b\in A(\RR^n_{>0})$, 
there is a bijection of sets
\[\VV_{>0}(H_{\k,b})  \to  \VV_{>0}(F_\k) /\T_A^{>0}\,,\quad x\mapsto x\circ\T_A^{>0}\,.\]
\end{proposition}

\begin{proof}
By \Cref{thm:invariance_lifts_to_closure}, $F$ is $\T_A$-invariant over $\RR_{>0}$. The statement follows from the classical result (in some settings known as \emph{Birch's theorem}) that for any $x,x^*\in\RR^n_{>0}$, the coset $x\circ \T_A^{>0}$ intersects the translated subspace $x^*+\ker(A)$ exactly once; see, e.g.,  \cite[Prop.~5.1 and~B.1]{Feinberg1995class} and \cite[Lem.~3.15]{Boros2012multiple} for a proof.
\end{proof}

\begin{example}
For the vertical system $F$ in \Cref{ex:square_network}, the number of $\T_A^{>0}$-cosets for $A=\begin{bmatrix}2 & 3\end{bmatrix}$ that make up $\VV_{>0}(F_\k)$
is found by counting the number of points in the intersection $\VV_{>0}(F_\k)\cap\VV(2x+3y-b)$ for any $b>0$, see \Cref{fig:toric_components2}(a-b).
\end{example}

\begin{figure}[t]
    \centering
      \begin{subfigure}[b]{0.3\textwidth}
\resizebox{\textwidth}{!}{\begin{tikzpicture}[scale=0.8]
  \draw[->] (-0.35,0) -- (4.2,0) node[right] {\scriptsize $x_1$};
  \draw[->] (0,-0.35) -- (0,4.2) node[above] {\scriptsize $x_2$};
  \draw[line width=1pt,scale=0.5, domain=0:2.35, smooth, variable=\x, specialblue] plot ({0.92*\x*\x}, {0.62*\x*\x*\x});
   \draw[line width=1.5pt,scale=0.5, domain=0:6, smooth, variable=\x, specialred] plot ({\x}, {-2/3*\x+4});
  \node[text width=2cm] at (4,3) {\scriptsize\color{specialblue} $\VV_{>0}(F_\k)$};
  \node[text width=2.2cm] at (3.0,1.2) {\scriptsize\color{specialred} $\VV_{>0}(2x+3y-1)$};
\end{tikzpicture}}
\caption{}
\end{subfigure}
\hspace{0.03\textwidth}
\begin{subfigure}[b]{0.3\textwidth}
    \resizebox{\textwidth}{!}{\begin{tikzpicture}[scale=0.8]
  \draw[->] (-0.35,0) -- (4.2,0) node[right] {\scriptsize $x_1$};
  \draw[->] (0,-0.35) -- (0,4.2) node[above] {\scriptsize $x_2$};
  \draw[line width=1pt,scale=0.5, domain=0:2.25, smooth, variable=\x, specialblue] plot ({0.83*\x*\x}, {0.76*\x*\x*\x});
  \draw[line width=1pt,scale=0.5, domain=0:2.35, smooth, variable=\x, specialblue] plot ({0.92*\x*\x}, {0.62*\x*\x*\x});
  \draw[line width=1pt,scale=0.5, domain=0:2.40, smooth, variable=\x, specialblue] plot ({0.97*\x*\x}, {0.55*\x*\x*\x});
  \draw[line width=1pt,scale=0.5, domain=0:6, smooth, variable=\x, specialred] plot ({\x}, {-2/3*\x+4});
  \node[text width=2cm] at (4,3) {\scriptsize\color{specialblue} $\VV_{>0}(F_\k)$};
  \node[text width=2.2cm] at (3.0,1.2) {\scriptsize\color{specialred} $\VV_{>0}(2x+3y-1)$};
\end{tikzpicture}
}
\caption{}
\end{subfigure}
\begin{subfigure}[b]{0.3\textwidth}
\def\svgwidth{1.1\linewidth}
\resizebox{0.9\textwidth}{!}{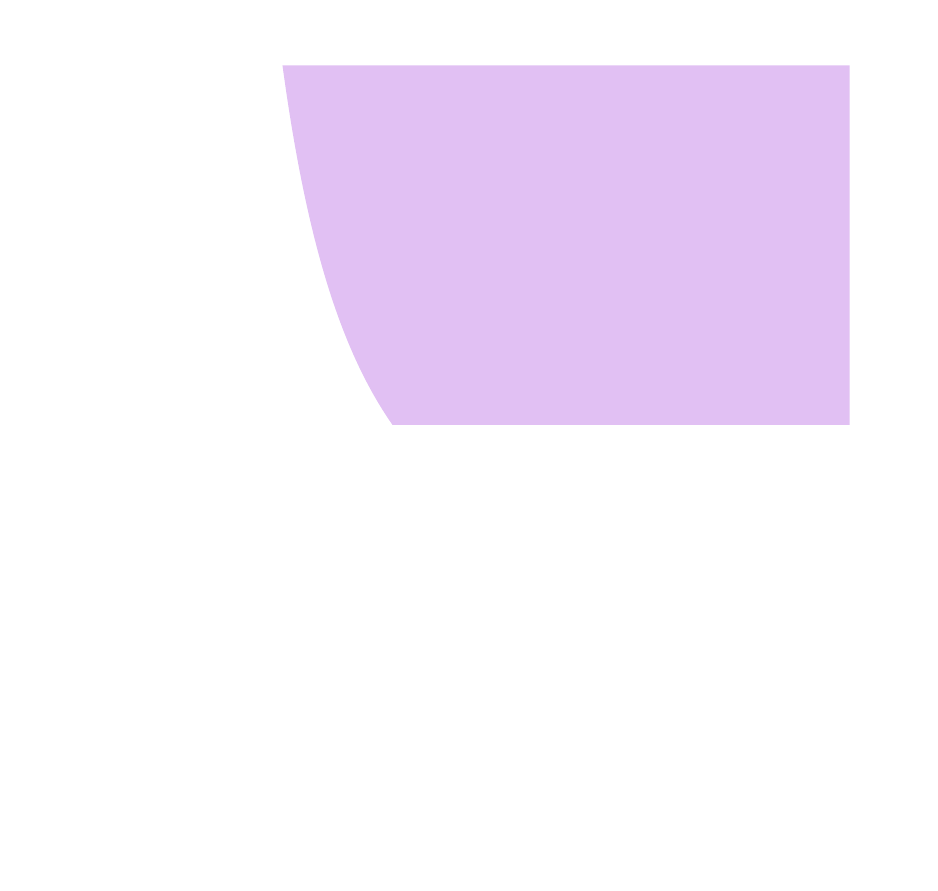}
\caption{}
\end{subfigure}
    \caption{
    (a)-(b) Positive zero locus $\VV_{>0}(F_\k)$ for the system in \Cref{ex:square_network} and geometric interpretation of the coset counting system, for different parameter values: (a)~$\k=(0.01,3,1,1)$ and (b)~$\k=(0.01,1,1,1)$.
 (c) The slice $\kappa_3=\kappa_4=1$ of the parameter space for the coset counting system.} 
    \label{fig:toric_components2}
\end{figure}

If $F$ is generically locally $\T_A$-toric over $\RR_{>0}$, then the coset counting system has generically a finite number of zeros and hence the generic number of cosets is bounded above by the mixed volume by Bernstein's theorem \cite{bernstein} (see also \cite[§7.5]{cox2005using}, as well as \cite{Gross2011degree} where Bernstein's theorem is applied to a similar system). Sharper bounds that take into account the dependencies among the parametric coefficients arise from Newton--Okunkov bodies \cite{obatake2022newton} and tropical methods \cite{HelminckRen2025,tropicalhomotopy}. This type of bounds can be quite far from the number of positive solutions, as they count solutions over $\CC^*$, but if the bound is $1$, then we conclude that $F$ is generically $\T_A$-toric.

In the next two subsections, we will discuss two other approaches to assert the system is toric,  that is, to confirm that the coset counting system has at most one solution for all $\k$.
We end this subsection with a short discussion on nondegeneracy of the coset counting system. 

\begin{lemma}
\label{lem:rank_of_coset_counting_jacobian}
Let $F$ be a vertical system that is $\T_A$-invariant for a matrix $A$ of full row rank, and let $H$ be the coset counting system \eqref{eq:system_number_of_cosets}. Let $\k\in(\CC^*)^m$ and $x^*\in\VV_{\CC^*}(F_\k)$ be such that $A\diag(x^*)A^\top$ is invertible (this is true for any $x^*\in\VV_{>0}(F_\k)$). Then
\[\rk(J_{H_{\k,Ax^*}}(x^*))=\rk(J_{F_{\k}}(x^*))+\rk(A)\,.\]
In particular, 
$x^*$ is a nondegenerate zero of $F_\k$ if and only if it is a nondegenerate zero of $H_{\k,Ax^*}$.
\end{lemma}
\begin{proof}
The assumed invertibility gives that we have a direct sum decomposition
\begin{equation}\label{eq:direct_sum}
\RR^n=\ker(A)\oplus \im(\diag(x^*) A^\top)\,.
\end{equation}
Furthermore, by invariance, it holds that $F(x^*\circ t^A)=0$ for all $t\in(\CC^*)^d$, which gives
\[0=\tfrac{\partial}{\partial t}F(x^*\circ t^A)\vert_{t=(1,\dots,1)}=J_{F_{\k}}(x^*)\diag(x^*)A^\top.\]
The lemma now follows by intersecting both sides of \eqref{eq:direct_sum} with $\ker(J_{F_{\k}}(x^*))$, which yields
\[\ker(J_{F_{\k}}(x^*))=\ker(J_{H_{\k,Ax^*}}(x^*))\oplus  \im(\diag(x^*) A^\top),\]
as $\ker(J_{H_{\k,Ax^*}}(x^*))=\ker(J_{F_{\k}}(x^*))\cap\ker(A)$.
\end{proof}

\Cref{lem:rank_of_coset_counting_jacobian} has the following interesting consequence for generic nondegeneracy.

\begin{proposition}
\label{prop:F_nondegenerate_iff_H_nondegenerate}
Let $F$ be a vertical system that is $\T_A$-invariant for a matrix $A$ of full row rank, and let $H$ be the coset counting system \eqref{eq:system_number_of_cosets}. Then $F$ is nondegenerate over $\CC^*$ if and only if $H$ is nondegenerate over $\CC^*$.
\end{proposition}

\begin{proof}
The reverse  implication is immediate, so we focus on  the forward implication. Nondegeneracy of $F$ implies that there is a nonempty Zariski open subset $\mathcal{U}$ of the incidence variety $\I=\{(\k,x)\in(\CC^*)^m\times(\CC^*)^n:F_\k(x)=0\}$ consisting of pairs $(\k,x)$ such that $x$ is a nondegenerate zero of $F_\k$. 
Since the projection $p\from\I\to(\CC^*)^n$ is surjective, it follows that $p(\mathcal{U})$ is a nonempty Zariski open subset of  $(\CC^*)^n$. By denseness,  $p(\mathcal{U})$  contains some $x^*\in (\CC^*)^n$ such that $A\diag(x^*) A^\top$ is invertible. The desired conclusion now follows from \Cref{lem:rank_of_coset_counting_jacobian}.
\end{proof}

\subsection{Injectivity}
Chemical reaction network theory has provided several methods to decide whether an augmented vertically parametrized system $(C(\k\circ x^M),Lx-b)$ has  two or more positive zeros for a choice of parameter values. 
One of the simplest methods decides whether the map $x\mapsto (C(\k\circ x^M),Lx)$ is injective on $\RR^n_{>0}$ for all  positive $\k$ (see \cite{wiuf2013determinant,Muller2015injectivity}), and  gives rise to the following sufficient  criterion {for a vertical system to be toric} over $\RR_{>0}$. Here, we formulate it in terms of a symbolic determinant, but we note that it can equivalently be formulated in terms of sign vectors, as discussed in \cite{Muller2015injectivity}.

\begin{theorem}\label{thm:injectivity}
Let $F$ be a vertical system with defining matrices $M\in\ZZ^{n\times m}$ and  $C\in\RR^{s\times m}$ of  rank $s$  with $\ker(C)\cap \RR^m_{>0}\neq \varnothing$. 
Assume that $F$ is $\T_A$-invariant for a full rank matrix $A\in\ZZ^{(n-s) \times n}$.  
For variables $\mu=(\mu_1,\dots,\mu_m)$ and $\lambda=(\lambda_1,\dots,\lambda_n)$, form the symbolic matrix 
\[{Q}_{\mu,\lambda}:=\begin{bmatrix}C\diag(\mu)M^\top\diag(\lambda)\\A\end{bmatrix}\,.\]
If $\det(Q_{\mu,\lambda})$ is a nonzero polynomial in $\RR[\mu,\lambda]$, with all nonzero coefficients having the same sign, 
 then $F$ is $\T_A$-toric over $\RR_{>0}$. 
\end{theorem} 
\begin{proof}
{It follows from \Cref{prop:vertical} combined with \Cref{prop:F_nondegenerate_iff_H_nondegenerate} that $F$ is nondegenerate over $\RR_{>0}$.}
By \cite[Thm.~2.13]{Muller2015injectivity}, the assumption on  $\det(Q_{\mu,\lambda})$ is equivalent to 
the polynomial map $F_{\k}$ being injective on $(x^*+\ker(A))\cap\RR^n_{>0}$ for all $x^*\in\RR^n_{>0}$ and $\k\in\RR^m_{>0}$. This implies that for the coset counting system $H$ in \eqref{eq:system_number_of_cosets}, $\#\VV_{>0}(H_{\k,b})\leq 1$ for all $\k$ and $b$, and the statement follows from \Cref{prop:coset_counting}.
\end{proof}

\begin{example}\label{ex:idhkp3}
The vertical system $F$ in \Cref{ex:IDH} satisfies the conditions in \Cref{thm:injectivity} for the matrix $A$   in \Cref{ex:IDH_invariance}, since 
\begin{align*}\det(Q_{\mu,\lambda})=&
-\lambda_{1}\lambda_{3}\lambda_{4}\mu_{1}\mu_{3}\mu_{4}-
\lambda_{1}\lambda_{4}\lambda_{5}\mu_{1}\mu_{4}\mu_{6}-
\lambda_{2}\lambda_{3}\lambda_{4}\mu_{1}\mu_{3}\mu_{4}\\[-0.5em]
&-\,\lambda_{2}\lambda_{4}\lambda_{5}\mu_{1}\mu_{4}\mu_{6}-
\lambda_{3}\lambda_{4}\lambda_{5}\mu_{2}\mu_{4}\mu_{6}
-\lambda_{3}\lambda_{4}\lambda_{5}\mu_{3}\mu_{4}\mu_{6}\,
\end{align*}
is a nonzero polynomial with all coefficients of the same sign. Hence
  $F$ is $\T_A$-toric over $\RR_{>0}$. 
\end{example}

\subsection{Constant number of cosets.}
In practice, it often holds that 
$\#(\VV_{>0}(F_\k)/\T_A^{>0})$ is constant with respect to $\k\in\RR^m_{>0}$. When this is the case, for all $\k$, $\#(\VV_{>0}(F_\k)/\T_A^{>0})$  can be inferred from 
the number of zeros of the coset counting system $H$ from  \eqref{eq:system_number_of_cosets} for any fixed  $\k\in\RR^m_{>0}$. 
 
The following is a sufficient criterion for $\#(\VV_{>0}(F_\k)/\T_A^{>0})$ to be constant for all $\k\in\RR^m_{>0}$. 
In intuitive terms, it rules out the three ways in which the number of zeros of $H$ can change as the parameters vary: 
(i) a zero crossing a coordinate hyperplane, (ii) a degenerate zero arising, or (iii) a zero going to infinity.

\begin{proposition}
\label{prop:constant_number_of_cosets}
Let $F$ be a vertical system with defining matrices $M\in\ZZ_{\geq 0}^{n\times m}$ and $C\in\RR^{s\times m}$ of rank $s$   with $\ker(C)\cap\RR^m_{>0} \neq \varnothing$. Assume that $F$ is $\T_A$-invariant  for a matrix $A\in\ZZ^{(n-s)\times n}$ of rank $n-s$. Suppose the following hold:\\[-0.8em]
\begin{enumerate}[label=(\roman*)]
\setlength{\itemsep}{0.2em}
    \item\label{it:no_boundary_zeros}  $\VV_{\RR}{(C(\k\circ x^M),Ax-b)}\cap(\RR^n_{\geq 0}\setminus\RR^n_{>0})=\varnothing$ for all $\k\in\RR^m_{>0}$ {and some $b\in A(\RR^n_{>0})$}.
   
    \item\label{it:nondegeneracy} 
    {$\rk(C\diag(w)M^\top) = s$ for all $w\in \ker(C)\cap \RR^m_{>0}$}.
    
    \item\label{it:compactness} $\row(A)\cap\RR^n_{>0}\neq\varnothing$.
\end{enumerate}
\vspace{0.2em}
Then $\#(\VV_{>0}(F_\k)/\T_A^{>0})$ is constant with respect to $\k\in\RR^m_{>0}$, and in particular $\Z_{>0} = \RR^m_{>0}$.
\end{proposition}

\begin{proof}
Let $H$ be the coset counting system from \eqref{eq:system_number_of_cosets} and consider the incidence correspondence 
\[\I=\{(\k,x)\in\RR^m_{>0}\times\RR^n_{\geq 0}:H_{\k, b}(x)=0\}\]
as well as the projection $\pi\from\I\to\RR^m_{>0}$ to parameter space. Condition (i) gives that $\I\subseteq\RR_{>0}^m\times\RR^n_{>0}$.
{Condition (ii) implies that $\rk(J_{F_\k}(x^*))=s$ for all $(\k,x^*)\in \I$
and hence, by \Cref{lem:rank_of_coset_counting_jacobian},  all zeros of $H_{\k,b}$ in $\RR^n_{>0}$ are nondegenerate.} It follows now from \cite[Prop.~3.3]{FeliuHenrikssonPascual2025vertical} that $\pi$ is an open map and lacks critical points.

We now prove that $\pi$ is surjective, i.e., that $\Z_{>0}=\RR_{>0}^m$. Since $\pi(\I)$ is nonempty and open, and $\RR^m_{>0}$ is connected, it suffices to prove that $\pi(\I)$ is closed. Condition (iii) gives that $\mathcal{P}:=\VV_\RR(Ax-b)\cap\RR^n_{\geq0}$ is compact \cite{benisrael}, which implies that the canonical projection $\RR^m_{>0}\times\mathcal{P}\to\RR^m_{>0}$ is a closed map. Since $\I$ is closed in $\RR^m_{>0}\times\mathcal{P}$, this shows that $\pi(\I)$ is closed.
 
Our next goal is to show that $\#\pi^{-1}(\k)$ is constant for all $\k$. As $\RR^m_{>0}$ is connected, it suffices to show that the cardinality is locally constant.  Given $\k^*\in \RR^m_{>0}$, write 
$\pi^{-1}(\k^*)=\{(\k^*,x_1),\ldots,(\k^*,x_\ell)\}$ for $x_1,\ldots,x_\ell\in\VV_{>0}(F_{\k^*})$. The absence of critical points for $\pi$ allows us to find, for each $i\in[\ell]$,  an open neighborhood $U_i\subseteq\I$ of $(\k^*,x_i)$ and an open neighborhood $V_i\subseteq\RR^m_{>0}$ of $\k^*$ such that $\pi_{\vert U_i}\from U_i\to V_i$ is a homeomorphism. 
These open sets can be chosen such that $U_i\cap U_j=\varnothing$ for $i\neq j$. Furthermore, 
by letting $V=\cap_{i=1}^\ell V_i$ and replacing $U_i$ by $(\pi_{\vert U_i})^{-1}(V)$, it holds that  $\pi(U_i)=V$ for all $i\in[\ell]$. 
Now, form the set $Q=\I\setminus(U_1\cup\cdots\cup U_\ell)$, which is closed in $\RR^m_{>0}\times \mathcal{P}$, which in turn gives that $\pi(Q)$ is closed in $\RR^m_{>0}$. 
It now holds that $W:= V\setminus\pi(Q)\subseteq\RR^m_{>0}$ is an open neighborhood of $\k^*$ such that $\#\pi^{-1}(\k)=\ell$ for all $\k\in W$.
\end{proof}

Condition \ref{it:no_boundary_zeros}  in \Cref{prop:constant_number_of_cosets} can efficiently be checked with SAT-SMT solvers \cite{smt-lib}. Alternatively, the theory of siphons from chemical reaction network theory provides a  sufficient condition \cite{AngeliLeenheerSontag2007,ShiuSturmfels2010}. Condition \ref{it:nondegeneracy} can be checked by finding a parametrization of $\ker(C)\cap\RR^m_{>0}$ via extreme rays of $\ker(C)\cap\RR^m_{\geq0}$ and evaluating a symbolic determinant.
Condition \ref{it:compactness} can be verified by deciding whether the polyhedral cone $\row(A)\cap\RR^n_{\geq 0}$ has nonempty interior. 

\begin{example}
\label{ex:trianglenetwork_revisited}
We use \Cref{prop:constant_number_of_cosets} to prove that $F$ in \eqref{eq:triangle} is toric over $\RR_{>0}$. 
From \Cref{alg:invariance} we find the maximal-rank matrix $A=\begin{bmatrix}2 & 3\end{bmatrix}$ for which $F$ is $\T_A$-invariant. 
It is immediate to see that conditions \ref{it:no_boundary_zeros} and \ref{it:compactness} in \Cref{prop:constant_number_of_cosets} hold. For \ref{it:nondegeneracy}, we find
\[\ker(C) \cap \RR^4_{>0} = \big\{\lambda_1(2, 0, 0, 1)+\lambda_2 (1, 1, 0, 0)+\lambda_3(0, 0, 2, 1)+\lambda_4(0, 1, 1, 0) : \lambda\in \RR^4_{>0}\big\}\,    \]
and hence the determinant  of any of the  matrices in condition \ref{it:nondegeneracy}
can be written as $- (9h_1+ 4h_2) ( 2\lambda_{1}+4 \lambda_{3}+ \lambda_{4})$ 
for some $\lambda\in \RR^4_{>0}$ and $h\in \RR^2_{>0}$. 
In particular, it does not vanish and condition \ref{it:nondegeneracy} holds. 
Next, we solve the coset counting system
\begin{equation}
\label{eq:triangle_coset}
H = \big( (\kappa_1-\kappa_2)x_1^3x_2^2+\kappa_3x_2^4-2\kappa_4x_1^6,\: 2x_1+3x_2 - 5 \big) \,
\end{equation}
numerically for $\k^*=(1,1,1,1)$ using \texttt{HomotopyContinuation.jl} \cite{homotopycontinuation} with certification \cite{breiding2020certifying}. The BKK bound of the system is 6, and we find 6 certifiably distinct zeros in $(\CC^*)^2$, of which precisely one is certifiably real and positive, whereas the other solutions are certifiably nonreal or nonpositive. Hence, we conclude that the system 
has a unique zero in $\RR^2_{>0}$, and \Cref{prop:constant_number_of_cosets} gives that
 $F$ is $\T_A$-toric, with $\Z_{>0}=\RR^4_{>0}$. 
In particular, 
 for each $\k\in\RR_{>0}^4$, the positive zero locus admits the monomial parametrization 
\[\RR_{>0}\to \VV_{>0}(F_\k)\,,\quad t\mapsto (\alpha_1 t^2,\alpha_2 t^3)\,,\]
where $\alpha=(\alpha_1,\alpha_2)$ is the unique positive zero of \eqref{eq:triangle_coset}. 
We point out that $F$ is not toric over $\RR^*$, since  $\VV_{\RR^*}(F_{\k^*})$ has two irreducible components. 
\end{example}

\subsection{Algorithm for $\GG=\RR_{>0}$}

Given matrices $C\in \RR^{s\times m}$ of rank $s$ and $M\in \ZZ^{n\times m}$, we have presented several results to address {whether the associated vertical system $F$ is (generically) toric}.
We gather these in \Cref{alg:summary}.
A Julia implementation of the algorithms is available in the GitHub repository
\begin{center}
    \url{https://github.com/oskarhenriksson/ToricVerticalSystems.jl}\,.
\end{center}
The implementation relies on the packages \texttt{Oscar.jl} \cite{OSCAR} for polyhedral and symbolic computations and \texttt{HomotopyContinuation.jl} \cite{homotopycontinuation} for certified numerical solving of the coset counting system.

\begin{algorithm}
\caption{Summary for $\RR_{>0}$}
\SetAlgoNoLine
\fontsize{10}{13}\selectfont
\KwIn{Matrices $C\in \RR^{s\times m}$ of  rank $s$,  $M\in\mathbb{Z}^{n\times m}$}
\KwOut{Whether $F=C(\k\circ x^M)$ is nondegenerate, (generically) locally $\T_A$-toric, or (generically) $\T_A$-toric over $\RR_{>0}$}

Run \Cref{alg:invariance}. Proceed if a matrix $A\in \ZZ^{d\times n}$ is returned (which implies $\ker(C)\cap \RR^m_{>0}\neq \varnothing$)

\# \textit{Decide nondegeneracy}

Generate a random $w\in \ker(C)$ and compute $r:=\rk(C \diag(w) M^\top)$

\If{$r<s$ and $\rk(C \diag(w) M^\top)<s$ for all $w\in \ker(C)$}{
    \Return{$F$ is degenerate}
}

\If{$s+d<n$}{
    \Return{$F$ is not generically locally $\T_A$-toric over $\RR_{>0}$}
}

\# \textit{At this point we know $F$ is generically locally $\T_A$-toric over $\RR_{>0}$}\\
\# \textit{We proceed to study the number of cosets}

\label{line:injectivity_test}
\If{All coefficients of $\det(Q_{\mu,\lambda})$ have the same sign}{
    \Return{$F$ is $\T_A$-toric over $\RR_{>0}$}
}

Find $\mathit{mv}:=$ mixed volume of the coset counting system

\If{$\rk(C \diag(w) M^\top)=s$ for all $w\in \ker(C)\cap \RR_{>0}^m$}{
    \If{$mv=1$}{
        \Return{$F$ is $\T_A$-toric over $\RR_{>0}$}
    }
    \If{conditions \ref{it:no_boundary_zeros}  and \ref{it:compactness} 
    in \Cref{prop:constant_number_of_cosets} hold}{
        Set $r:=$ number of solutions in $\RR^n_{>0}$ of the coset counting system for a random $\k\in \RR^m_{>0}$
        
        \If{$r=1$}{
            \Return{$F$ is $\T_A$-toric over $\RR_{>0}$}
        }
        \Return{$F$ is locally $\T_A$-toric over $\RR_{>0}$ with $r$ cosets}
    }
    \Return{$F$ is locally $\T_A$-toric over $\RR_{>0}$ with 
    at most $\mathit{mv}$ number of cosets}
}

\If{$\mathit{mv}=1$}{
    \Return{$F$ is generically $\T_A$-toric over $\RR_{>0}$}
}

\Return{$F$ is generically locally $\T_A$-toric over $\RR_{>0}$ with generically {at most} $\mathit{mv}$ cosets}
\label{alg:summary}
\end{algorithm}

\section{Reaction-network-theoretic perspectives {on toric invariance}}
\label{sec:crnt_perspectives}

In this final section, we focus on the motivating scenario, namely that of reaction networks.  We introduce them and their vertical steady state systems in \Cref{subsec:crnt_intro}, and show how (local) toric structure can be used to analyze them in \Cref{subsec:crnt_applications}. We  discuss model reductions via intermediates in \Cref{subsec:intermediates},  and apply our algorithms to a database of biological networks 
in \Cref{subsec:odebase}.
Finally, we compare our criteria to previous  work 
in \Cref{subsec:other_flavors}. 

\subsection{Reaction networks}
\label{subsec:crnt_intro}
A \term{reaction network} on an ordered set $\S=\{X_1,\ldots,X_n\}$ of species is a collection of $m$ reactions between formal nonnegative linear combinations of the species (called complexes):
\begin{equation}\label{eq:reaction_network}\sum_{i=1}^n \alpha_{ij} X_i \longrightarrow \sum_{i=1}^n \beta_{ij} X_i\,,\qquad j\in [m]\,,\end{equation}
where $\alpha_{ij},\beta_{ij}\in\ZZ_{\geq 0}$. 
The net production of the species in each of the reactions is encoded by the \term{stoichiometric matrix} $N=(\beta_{ij}-\alpha_{ij})\in\ZZ^{n\times m}$.

The concentration of the respective species is denoted by a vector $x=(x_1,\ldots,x_n)\in\RR^n_{\geq 0}$. 
Under common assumptions,
these concentrations vary according to an autonomous ordinary differential equation system of the form
\begin{equation}
\label{eq:ODE}
\tfrac{dx}{dt}=N(\k\circ x^M)\,, \qquad x\in \RR^n_{\geq 0}\, ,\end{equation}
where   $M  \in \ZZ^{n\times m}$ is called the kinetic matrix and $\k=(\kappa_1,\ldots,\kappa_m)\in\RR^m_{>0}$ is a vector of rate constants, which typically are viewed as unknown parameters. 

The main example of this construction arises under the mass-action assumption, where $M=(\alpha_{ij})$ is the reactant matrix consisting of the coefficients of the left-hand sides of the reactions. In this case, $\RR^n_{\geq 0}$ is forward-invariant by the ODE system \eqref{eq:ODE}.

Letting $s=\rk(N)$, we  choose a matrix $C\in \RR^{s\times m}$ of rank $s$ such that $\ker(C)=\ker(N)$. Then, the steady states of \eqref{eq:ODE} are the zeros of the vertically parametrized system 
\[F = C(\k\circ x^M)\in\RR[\k,x^\pm]^s \, .\] 
We call any such system \term{the steady state system} (with a choice of $C$   implicitly made). 
One is particularly interested in steady states with \emph{strictly positive} entries.

The trajectories of \eqref{eq:ODE} are confined in parallel translates of $\im(N)$, which can be written as
\begin{equation}\label{eq:classes}
\{x\in\RR^n_{\geq 0}:Lx- b=0\},\qquad b\in \RR^{n-s}
\end{equation}
for a matrix
$L\in \RR^{(n-s)\times n}$ whose rows form a basis for the left-kernel of $N$. 
We call such a matrix $L$ a \term{conservation law matrix}. 
The sets in \eqref{eq:classes} are called \term{stoichiometric compatibility classes}.  The steady states in a given stoichiometric compatibility class are therefore the positive zeros of the augmented vertical  system 
\begin{equation}\label{eq:augmented_crn}
(C(\k\circ x^M), Lx-b)\in\RR[\k,b,x^\pm]^n \, .
\end{equation}

 \begin{figure}[t]
   \centering
\begin{subfigure}[b]{0.28\textwidth}
  {\small   \begin{align*}
X_1 + X_2 & \ce{<=>[\kappa_1][\kappa_2]} X_3 \ce{->[\kappa_3]} X_1 + X_4\\
X_3 +X_4 & \ce{<=>[\kappa_4][\kappa_5]} X_5 \ce{->[\kappa_6]} X_2 + X_3\,
   \end{align*}}
\caption{Network for \Cref{ex:IDH}}
\end{subfigure}
~
\begin{subfigure}[b]{0.28\textwidth}
{\small \[\begin{tikzcd}[column sep=small,ampersand replacement=\&]
    6X_1 \arrow[rd,end anchor={[yshift=-0.3ex,xshift=0.3ex]},"\kappa_{4}"'] \&  \& 3X_1 + 2X_2
    \arrow[ld, rightharpoondown, shift right=0.3ex,end anchor={[yshift=-0.3ex,xshift=-0.3ex]},"\kappa_{2}"',outer sep=-0.3ex] \arrow[ll,"\kappa_{1}"'] \\  \& 4X_2 \arrow[ru, rightharpoondown, shift right=0.3ex,start anchor={[yshift=-0.3ex,xshift=-0.3ex]},"\kappa_{3}"',outer sep=-0.3ex] \& 
    \end{tikzcd}\]}
\caption{Network for \eqref{eq:triangle}}
\end{subfigure}
~
\begin{subfigure}[b]{0.3\textwidth}
{\small  \[\begin{tikzcd}[ampersand replacement=\&]
	9X_1 \& 3X_1 + 4X_2 \\
	6X_1 + 2X_2 \& 6X_2\,.
	\arrow["{\kappa_{1}}", from=1-1, to=1-2]
	\arrow["{\kappa_{2}}", from=1-2, to=2-2]
	\arrow["{\kappa_{3}}"', from=2-2, to=2-1]
	\arrow["{\kappa_{4}}", from=2-1, to=1-1]
\end{tikzcd}\]}
\caption{Network for \Cref{ex:square_network}}
\end{subfigure}
    \caption{(a)  is a model of the IDHKP-IDH system in bacterial cell \cite[Fig. 3]{Shinar2010structural}; (b)  is a variation of the classical triangle network that appears in several places in the literature (e.g., \cite[Eq.~7-2]{Horn1972general}, \cite[Ex.~1]{Craciun2009toric},  \cite[Ex.~2.3]{Millan2012toricsteadystates}); (c) 
     is a variation of a classical network studied in \cite[§7]{Horn1972general}.} 
    \label{fig:networks}
\end{figure}

Many of the examples of vertical systems in the previous sections are steady state systems of reaction networks with mass-action kinetics; see \Cref{fig:networks}.

\subsection{Consequences of toric invariance}
\label{subsec:crnt_applications}
{As alluded to in the introduction, the following {three} algebraic properties play a central role in the 
(algebraic)
study of reaction networks:}
\begin{itemize}
    \item The network is said to have the capacity for \term{multistationarity} if system \eqref{eq:augmented_crn} admits at least two positive zeros for some choice of $\k\in\RR^m_{>0}$ and $b\in \RR^{n-s}$. 
    \item A network is said to have \term{absolute concentration robustness} (or \term{ACR} for short) with respect to {$X_i$} if     
    $\pi_i(\VV_{>0}(F_\k))$ consists of at most a single point for all $\k\in\RR^m_{>0}$, where $\pi_i\from\RR^n\to\RR$ denotes the canonical projection onto the $i$th factor.
    \item A polynomial $f\in \RR(\k)[x^\pm]$ 
   is said to be a  \term{steady state invariant} if for all $\k\in \RR^m_{>0}$, $f$ specializes to $\k$ and  $\VV_{>0}(F_\k) \subseteq \VV_{>0}(f_\k)$. 
\end{itemize}

\smallskip

See \Cref{fig:crnt_applications} for illustrations. These three problems have in common that they become much easier to address under the assumption of toric invariance,  only by knowing the exponent matrix $A$
(which we have seen is easy to determine). We now treat each property separately.

\begin{figure}
\definecolor{varietyblue}{RGB}{0,90,170}
\definecolor{classred}{RGB}{126,12,12}
\definecolor{planegreen}{RGB}{12,100,12}
    \centering
    \begin{subfigure}[b]{0.3\textwidth}
    \centering
    \begin{tikzpicture}
    \node[anchor=south west, inner sep=0] (img) at (0,0)
    {\includegraphics[height=0.7\textwidth]{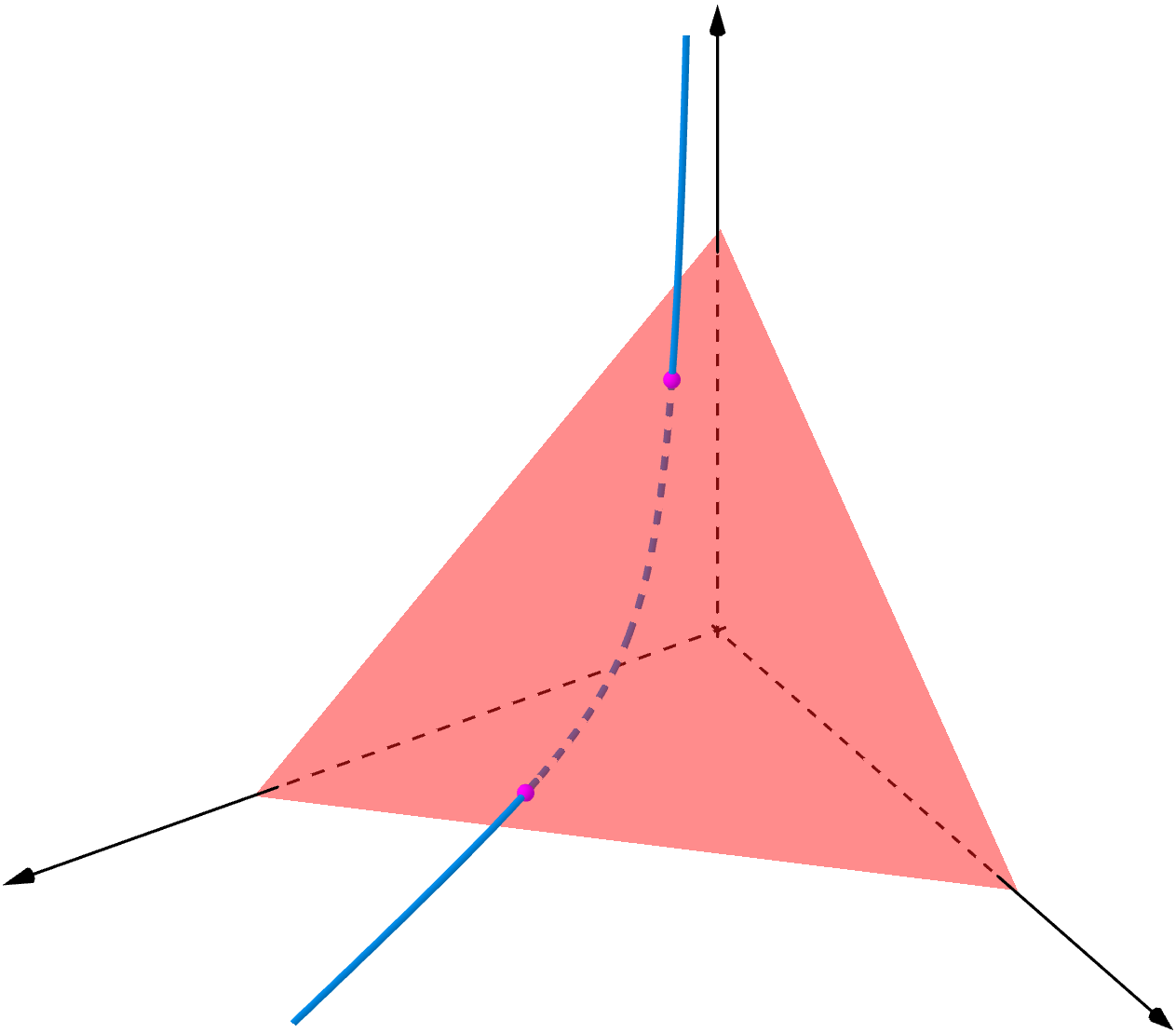}};
    \begin{scope}[x={(img.south east)}, y={(img.north west)}]
    \node[font=\tiny, varietyblue] at (0.4,0.9) {$\VV_{>0}(F_\k)$};
    \node[font=\tiny, anchor=west, classred] at (0.0,0.6) {$\VV_{>0}(Lx-b)$};
    \end{scope}
    \end{tikzpicture}
    \caption{Multistationarity}
    \end{subfigure}
    \begin{subfigure}[b]{0.3\textwidth}
    \centering
    \begin{tikzpicture}
    \node[anchor=south west, inner sep=0] (img) at (0,0)
    {\includegraphics[height=0.7\textwidth]{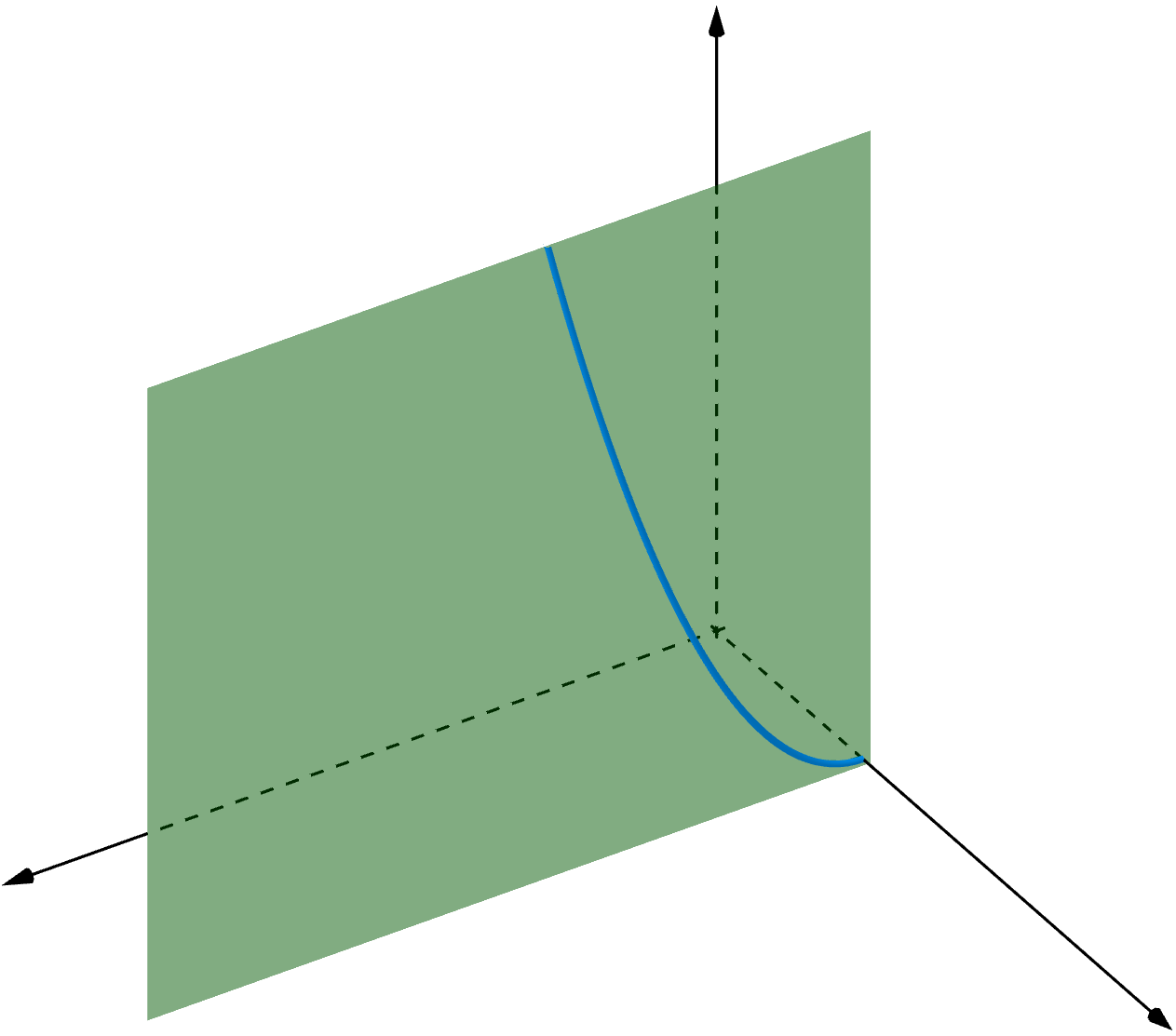}};
    \begin{scope}[x={(img.south east)}, y={(img.north west)}]
    \node[font=\tiny, varietyblue] at (0.3,0.8) {$\VV_{>0}(F_\k)$};
    \node[font=\tiny, anchor=west, planegreen] at (0.4,0.06) {$\VV_{>0}(x_i-c_\k)$};
    \end{scope}
    \end{tikzpicture}
    \caption{ACR}
    \end{subfigure}
    \begin{subfigure}[b]{0.3\textwidth}
    \centering
    \begin{tikzpicture}
    \node[anchor=south west, inner sep=0] (img) at (0,0)
    {\includegraphics[height=0.5\textwidth]{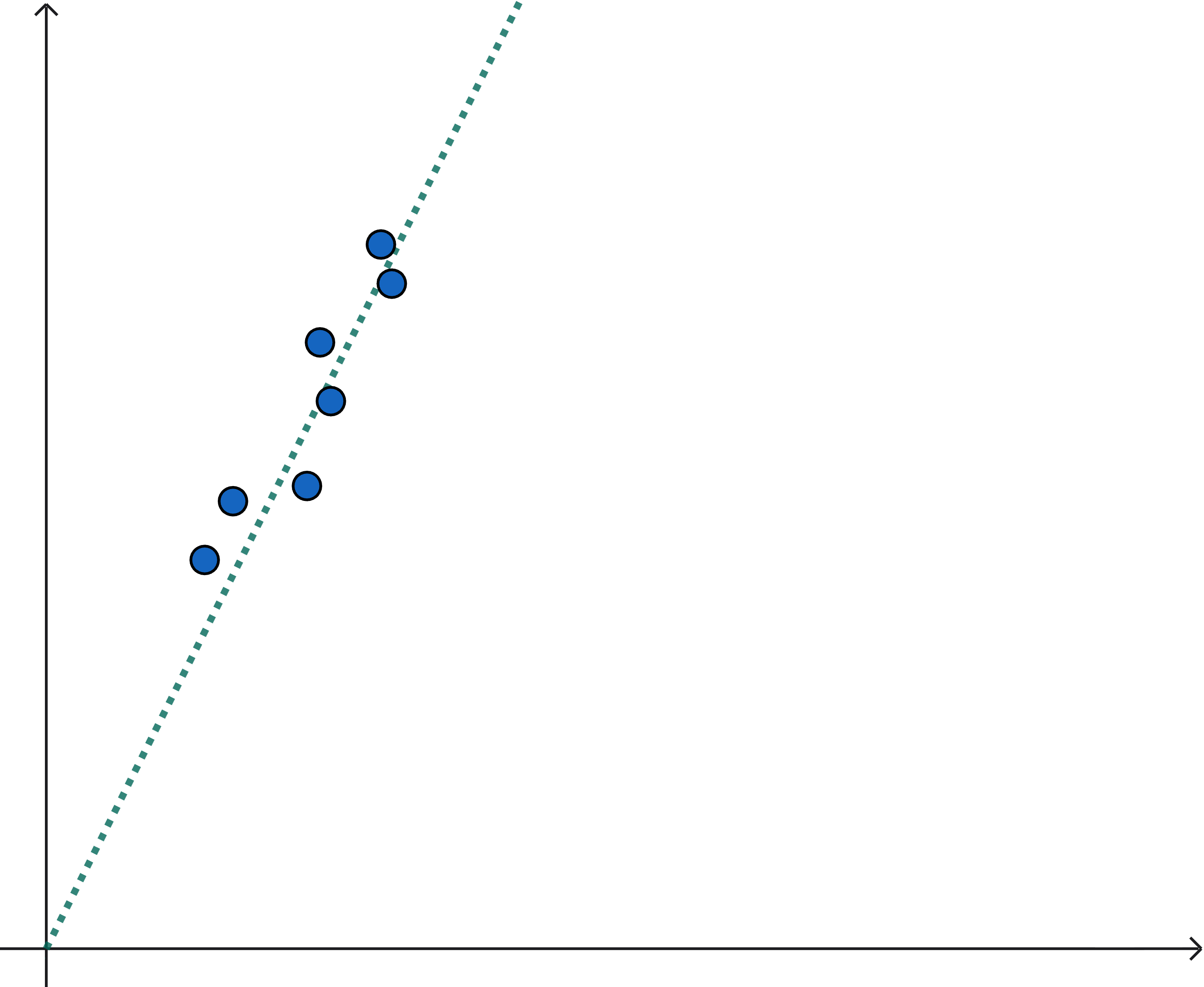}};
    \begin{scope}[x={(img.south east)}, y={(img.north west)}]
    \node[font=\tiny, anchor=south] at (0.02,0.98) {$x^{u_1}$};
    \node[font=\tiny, anchor=west] at (0.98,0.05) {$x^{u_2}$};
    \end{scope}
    \end{tikzpicture}
    \vspace{0.8em}
    \caption{Steady state invariants}
    \end{subfigure}
    \caption{Schematic illustration of three applications of toric structure for reaction networks.}
    \label{fig:crnt_applications}
\end{figure}

\smallskip
\noindent
\textit{Multistationarity.} Multistationarity might imply that trajectories converge to different steady states for the same parameter values and has been associated with robust cellular decision making. 
Deciding upon multistationarity corresponds to the general problem of determining whether an augmented vertical system has multiple positive zeros for some choice of parameter values. This is, in general, a challenging quantifier elimination problem. However, under the assumption of toric invariance, it simplifies to a sign condition.

The following proposition is a consequence and extension of \cite[§5]{Millan2012toricsteadystates} and \cite[§3]{Muller2015injectivity} (see also \cite[§2]{sadeghimanesh:multi}) from toric to torically invariant vertical systems.  
In particular, we provide a sufficient condition for the existence of multiple positive zeros that only relies on $\T_A$-invariance. 
For a set $P\subseteq\RR^n$, we let $\sign(P)$ denote the set of tuples in $\{0,+,-\}^n$ obtained by taking the sign of all elements of $P$ componentwise. 

\begin{proposition}[Multiple zeros]
\label{prop:multistationarity}
Let $F$ be a vertical system with defining matrices $M\in\ZZ^{n\times m}$  and $C\in\RR^{s\times m}$ of rank $s$ with $\ker(C)\cap\RR^m_{>0}\neq\varnothing$. Let  $L\in \RR^{(n-s)\times n}$.
Assume that $F$ is $\T_A$-invariant for a  matrix $A\in\ZZ^{d\times n}$ of rank $d$. Let $B\in \RR^{n\times (n-d)}$  be a full rank matrix whose columns form a basis for $\ker(A)$. 
For variables $\lambda=(\lambda_1,\dots,\lambda_{n})$ form the symbolic matrix
\[\Gamma_{\lambda}:=\begin{bmatrix}  B^\top \diag(\lambda) \\ L \end{bmatrix}\,.\]
Then for the following statements it holds that $\text{(i)}\Leftrightarrow\text{(ii)}\Leftrightarrow\text{(iii)}\Rightarrow\text{(iv)}$: 
\begin{enumerate}[label=(\roman*)]
  \item  $\ker(\Gamma_{\lambda})\neq \{0\}$ for some $\lambda\in\RR^n_{>0}$.
    \item $\sign(\ker(B^\top)) \cap \sign(\im(L)^\perp)\neq \{0\}$.
     \item The map $x\mapsto x^B$ is not injective on some coset $(x^*+\im(L)^\perp)\cap\RR^n_{>0}$ for $x^*\in\RR^n_{>0}$. 
    \item {$\#\VV_{>0}(C(\k\circ x^M) ,  Lx-b )\geq 2$ for some $\k\in \RR^m_{>0}$ and $b\in \RR^{n-s}$.   }
\end{enumerate}
Furthermore, if $F$ is $\T_A$-toric over $\RR_{>0}$, then also $\text{(iv)}\Rightarrow\text{(iii)}$.
\end{proposition}

\begin{proof}
The equivalence of (i), (ii) and (iii) follows from (the proof of) Theorem 2.13, Proposition 3.9 and Corollary 3.11 in  \cite{Muller2015injectivity} (where $x^B$ refers to the monomial map with exponents given by the \emph{rows} of $B$). 
Statement (iii) holds if and only if there exist distinct $x,y\in\RR^n_{>0}$ with  $Lx=Ly$ and 
$x^{B}=y^B$, which by \Cref{rmk:equivalent_descriptions_of_cosets} is equivalent to the existence of distinct $x,y\in\RR^n_{>0}$ with   $Lx=Ly$ and $y\in x\circ\T_A$. 
By taking $\k$ such that $x\in \VV_{>0}(F_\k)$, $\T_A$-invariance gives the implication (iii) to (iv).  {If $F$ is $\T_A$-toric}, the reverse implication also holds. 
\end{proof}

\begin{remark}
 When $d=n-s$, the matrix $\Gamma_{\lambda}$ is square,  and condition (i) in \Cref{prop:multistationarity}  holds if and only if $\det(\Gamma_{\lambda})$ vanishes for some $\lambda\in\RR^n_{>0}$. This is true precisely if $\det(\Gamma_{\lambda})$, viewed as a polynomial in $\RR[\lambda]$, is either zero or has two nonzero terms with different signs \cite[Thm.~2.13]{Muller2015injectivity}. Condition (ii) can be checked by deciding the feasibility of a system of linear inequalities, see \cite[§4]{Muller2015injectivity}. 
\end{remark}

\smallskip
\noindent
\textit{Absolute concentration robustness.} ACR means that  the concentration $x_i$ at steady state is independent from initial conditions of the system. Because of this, ACR is believed to be a mechanism that contributes to the remarkable robustness many biological systems display to changes in their environment \cite{Shinar2010structural}.

The problem of ACR has been studied with various algebraic techniques   \cite{Millan2011thesis,KarpPerezDasguptaDickensteinGunawardena2012,Pascualescudero2022local,ACR:manyauthors,feliu:dimension}. 
In general, this requires determining whether the elimination ideal $I(\VV_{>0}(F_\k))\cap\RR[x_i^{\pm}]$   has a linear polynomial {$x_i-c_\k$ for some $c_\k\in\RR_{>0}$} for all $\k$, 
which is a highly challenging computational problem \cite{ACR:manyauthors}.

A weaker notion is that of \term{local ACR}, where one requires $\VV_{>0}(F_\k)$ being contained in finitely many  translates of the $i$th coordinate hyperplane  \cite{Pascualescudero2022local}, corresponding to $I(\VV_{>0}(F_\k))\cap\RR[x_i^\pm]$ 
being nonzero for all $\k\in\RR^m_{>0}$.  A simpler problem is to decide whether local ACR arises generically in $\k$, which can be addressed using the ideal generated by $F$, see \cite{feliu:dimension}.

Under the hypothesis of $\T_A$-invariance,   constant coordinates in $\VV_{>0}(F_\k)$ can be easily read from the matrix $A$, as the following result formalizes. 
 
\begin{proposition}[Constant coordinates]\label{prop:acr}
Let $F$ be a vertical system with defining matrices $M\in\ZZ^{n\times m}$ and $C\in\RR^{s\times m}$ of rank $s$    with $\ker(C)\cap\RR^m_{>0}\neq\varnothing$. Assume that $F$ is $\T_A$-invariant  for a  matrix $A\in\ZZ^{d\times n}$ of rank $d$. 
The following holds: 
\begin{enumerate}[label=(\roman*)]
    \item If  for all $\k\in \RR^m_{>0}$, $\VV_{>0}(F_\k)$ is contained in a finite union of translates  of the coordinate hyperplane $\{x_i=0\}$, then $A_i$ is zero.
    
    \item If  $F$ is locally $\T_A$-toric over $\RR_{>0}$, then  the converse of (i) holds.
       \item Assume $F$ is $\T_A$-toric over $\RR_{>0}$. Then  $\VV_{>0}(F_\k)$ is  contained in  a  translate of the coordinate hyperplane $\{x_i=0\}$ for all $\k\in \RR^m_{>0}$ if and only if $A_i$ is zero.
    
    \end{enumerate}
\end{proposition}

\smallskip
\noindent
\textit{Steady state invariants.}
Steady state invariants may be used to discriminate between different  reaction networks that may be underlying a mechanism. The ideas behind this approach were introduced in  \cite{gunawardena-invariants,ManraiGunawardena2008} and further developed   for example in \cite{HarringtonHoThorneStumpf2012,KarpPerezDasguptaDickensteinGunawardena2012}.  

For a subset of indices $J=\{j_1,\dots,j_r\}\subseteq[n]$, which correspond to experimentally measurable concentrations, assume we have  
a sample of the observable variables at steady state, obtained from a fixed but unknown $\k$. Each sample point corresponds then to a point of $\VV_{>0}(F_\k)$. 

 Now assume 
there is a steady state invariant $f$ that only depends on 
$x_J=(x_{j_1},\dots,x_{j_r})$ and has support consisting of the monomials $x_J^{u_1},\dots,x_J^{u_p}$ for $u_1,\ldots,u_p\in\ZZ^r$. We can then test the hypothesis that the sample points come from $\VV_{>0}(F_\k)$ by statistically checking whether the image of the sample under the map
$\hat{x} \mapsto (\hat{x}^{u_1},\dots,\hat{x}^{u_p})$ lies on a hyperplane in $\RR^p$ through the origin; see \Cref{fig:crnt_applications}(c). 
This is a parameter-free method, that does not require knowing $\k$ nor the coefficients of the invariant $f$, but instead only its support.
 
The goal is thus to find a tuple 
of monomials $(x_J^{u_1},\dots,x_J^{u_p})$ that constitutes the support of a polynomial in $I(\VV_{>0}(F_\k))\cap\RR[x_J^\pm]$ for all $\k\in \Z_{>0}$. 
A common method is to study the elimination ideal
$\langle F \rangle \cap\RR(\k)[x_J^\pm]$ with Gröbner bases methods. However, this will give polynomials vanishing on the entire complex varieties $\VV_{\CC^*}(F_\k)$, and hence might miss useful invariants that only vanish on the positive parts (cf. \Cref{ex:missed_invariants} below).
 
Under the assumption of $\T_A$-invariance, the analysis simplifies substantially to studying the kernel of the submatrix $A_J$ consisting of the columns of $A$ with index in $J$.
The key observation is that 
$I(\alpha\circ \T_A^{>0})=\langle x^{u}-\alpha^u :u\in\ker_{\ZZ}(A)\rangle$, 
from which we obtain
\begin{equation}\label{eq:elimination}
I(\alpha\circ \T_A^{>0})\cap\RR[x_J^\pm]=I(\alpha_J\circ\T_{A_J}^{>0})=\big\langle x_J^{u}-\alpha_J^u :u\in\ker_{\ZZ}(A_J)\big\rangle\, . 
\end{equation}
Hence, if $\VV_{>0}(F_\k)$ is  a disjoint union of 
$\T_A^{>0}$-cosets, then  $I(\VV_{>0}(F_\k))\cap\RR[x_J^\pm]$ is the intersection of one ideal of the form \eqref{eq:elimination} per coset.
This gives rise to the following result (which, incidentally, recovers \Cref{prop:acr} about ACR when $J$ is a singleton).

\begin{proposition}[Polynomial relations]
\label{prop:ssinvariants}
Let $F$ be a vertical system that is $\T_A$-invariant for a matrix $A\in\ZZ^{d\times n}$ of rank $d$ and let $J\subseteq [n]$. Then the following holds:
\begin{enumerate}[label=(\roman*)]
    \item If $I(\VV_{>0}(F_\k))\cap\RR[x_J^\pm]\neq \{0\}$ for some $\k\in\RR^m_{>0}$, then $\ker(A_J)\neq \{0\}$.
    \item Suppose $\ker(A_J)\neq\{0\}$. Then  $I(\VV_{>0}(F_\k))\cap\RR[x_J^\pm]\neq \{0\}$ for any $\k$ such that $\VV_{>0}(F_\k)$ is locally $\T_A$-toric over $\RR_{>0}$. 
\end{enumerate}
Under the assumption that $F$ is $\T_A$-toric and that $\Z_{>0}\neq\varnothing$, the following additional points hold:
\begin{itemize}
    \item $I(\VV_{>0}(F_\k))\cap\RR[x_J^\pm]\neq \{0\}$ for all $\k\in \Z_{>0}$ if and only if $\ker(A_J)\neq\{0\}$. 
    \item For all nonzero vectors $u\in \ker_{\ZZ}(A_J)$  and $\k\in \Z_{>0}$, $I(\VV_{>0}(F_\k))\cap\RR[x_J^\pm]$ contains a binomial with monomials $x_J^u$ and $1$.  
\end{itemize}
\end{proposition}

Recall that the minimal nonempty subsets $J\subseteq [n]$ for which $\ker(A_J)\neq\{0\}$ are precisely the circuits of the column matroid $\mathcal{M}_A$. 

For steady state systems that are $\T_A$-toric, the model discrimination idea outlined above reduces to checking statistically the hypothesis that the monomial $x_J^u$ is constant at the sample points for appropriate nonzero vectors $u\in\ker_{\ZZ}(A_J)$. If the system is \emph{locally} $\T_A$-toric and $K$ is an upper bound on the number of cosets (e.g., the mixed volume of the coset counting system), the hypothesis to test is instead whether $x_J^u$ attains at most $K$ distinct values at the sample.
For the best numerical and statistical performance for noisy data, one should use vectors in $\ker_{\ZZ}(A_J)$ of small length, found systematically via, e.g., the Lenstra–Lenstra–Lovász algorithm \cite{LenstraLenstraLovasz1982}.

\smallskip

We end the section by studying multistationarity, ACR and invariants for three examples.

\begin{example}\label{ex:IDH_multi}
\Cref{ex:IDH}, which we know is toric, corresponds to the network displayed in \Cref{fig:networks}(a). We apply \Cref{prop:multistationarity} with the following matrix of conservation laws $L$ and matrix $B$ whose columns form a basis of $\ker(A)$ for $A$ as in   \Cref{ex:IDH}:
\[L= {\small\begin{bmatrix} 1 & 0 & 1 & 0 & 1\\ -2 & 1 & -1 & 1 & 0\end{bmatrix}} 
\quad\text{and}\quad 
B={\small\begin{bmatrix} -1 & 0 & -1\\ -1 & 0 & -1\\ 0 &0 & 1\\ 0 &1 & 0\\ 1&0&0\end{bmatrix}} .\]
This gives
\[\det(\Gamma_{\lambda})=\det\left[\begin{smallmatrix}  -\lambda_{{1}}&-\lambda_{{2}}&0&0&\lambda_{{5}}
\\ 0&0&0&\lambda_{{4}}&0\\ -\lambda_
{{1}}&-\lambda_{{2}}&\lambda_{{3}}&0&0 \\
1&0&1&0&1 \\
-2&1&-1&1&0
\end{smallmatrix}\right]=-\lambda_{{1}}\lambda_{{3}}\lambda_{{4}}-\lambda_{{1}}\lambda_{{4}}\lambda_{
{5}}-2\,\lambda_{{2}}\lambda_{{3}}\lambda_{{4}}-\lambda_{{2}}\lambda_{{4}}\lambda_{{5}}-\lambda_{{3}}\lambda_{{4}}\lambda_{{5}}\,.
\]
As this determinant does not vanish for $\lambda\in \RR^5_{>0}$, condition (i) in \Cref{prop:multistationarity} holds, and hence the augmented system in (iv) does not have multiple positive zeros.

As the  fourth column of $A$ is zero, the network displays ACR with respect to $X_4$ by  \Cref{prop:acr}(iii).
Suppose now that the index set of the observable variables is $J=\{1,2,3\}$.  Since $u=(1,1,-1)\in\ker_{\ZZ}(A_J)$, it follows from \Cref{prop:ssinvariants} that we can assess the validity of the model by checking whether $x^u=x_1x_2x_3^{-1}$ 
has a constant value at a sample of points. Since $J$ is a circuit of $\mathcal{M}_A$, there is no invariant that involves a proper subset of these variables. 
\end{example}

\begin{example}\label{ex:calcium_network}
Consider the following network  studied in \cite[Eq. 2]{Straube2013reciprocal}:
\begin{align*}
X_1+X_2 &\ce{<=>} X_3 \ce{->} X_2+X_4 & X_7+X_8 &\ce{<=>} X_2 \\
X_4+X_5 &\ce{<=>} X_6 \ce{->} X_1+X_5 & X_5+X_8 &\ce{<=>} X_9\, .
\end{align*}
We apply \Cref{alg:summary} 
and conclude that the steady state system is $\T_A$-toric with 
\[A={\small\begin{bmatrix} 1 & 0 & 1 & 1 & 0 & 0 & 0 & 1 & 0\\
-1 & 1 & 0 & 0 & 1 & 0 & 0 & 0 & 0\\
-1 & 1 & 0 & 0 & 0 & 1 & 0 & 0 & 1\\
1 & 0 & 1 & 0 & 0 & 0 & 1 & 1 & 1
\end{bmatrix}}.\]
In particular, it passes the injectivity test. (It also satisfies the conditions in \Cref{prop:constant_number_of_cosets}, and a certified numerical computation reveals that the number of cosets is 1.) 
By \Cref{prop:acr}, the network does not have ACR. 
The polynomial $\det(\Gamma_\lambda)$ in \Cref{prop:multistationarity} has both positive and negative coefficients and hence the network has the capacity for multistationarity. 
{For $J=\{3,8\}$, $(1,-1)\in\ker_{\ZZ}(A_J)$ and hence  $x_3x_8^{-1}$ is constant on  $\VV_{>0}(F_\k)$ for every $\k\in\RR^{10}_{>0}$.}
\end{example}

\begin{example}
\label{ex:missed_invariants}
Consider the vertical system $F$ from \eqref{eq:triangle}, which is $\T_A$-toric over $\RR_{>0}$ for $A=[\,2\:\:3\,]$ and  corresponds to the network in \Cref{fig:networks}(b). 
A simple calculation shows that the network does not  have ACR nor the capacity for multistationarity. By \Cref{prop:ssinvariants},  $I(\VV_{>0}(F_\k))$ contains a binomial invariant with supports $x_1^{3}x_2^{-2}$ and $1$ for all $\k\in\RR^4_{>0}$. However,  the ideal $\langle F\rangle$ in $\RR(\k)[x^\pm]$ contains no binomial. This simple example illustrates that  \Cref{prop:ssinvariants} allows us to find steady state invariants that are inaccessible from  the ideal of the steady state system.
\end{example}

\subsection{Network reduction and {toric invariance}}
\label{subsec:intermediates}
For a reaction network with mass-action kinetics, 
the search for toric invariance of the steady state system can be  simplified by removing  \emph{single-input intermediates}, a concept that we review now. 
For a more detailed presentation we refer to \cite{Feliu2013intermediates,sadeghimanesh2019grobner,sadeghimanesh:multi}. 

Given a network   with set of species $\S$,  a \term{choice of intermediates} is a partition $\S=\X\sqcup\Y$ of the set of species into a set of \term{non-intermediates} $\X$ and a set of \term{intermediates} $\Y$, with the following properties:
\begin{enumerate}[label=(\roman*)]
    \item Each species $Y\in\Y$ only appears in  complexes where the coefficients sum to 1.
    \item For each $Y\in\Y$ there exists a sequence of reactions 
    \begin{equation}\label{eq:sec_intermediates}
    c\to Y_1\to  \cdots\to Y_i\to Y \to Y_{i+1} \to \cdots \to Y_r \to c'\end{equation}
 with $Y_1,\ldots,Y_r\in\Y\setminus\{Y\}$ (there might be repetitions) and $c,c'$  are complexes in the non-intermediates $\X$. 
    \end{enumerate}
The non-intermediate complex $c$ in \eqref{eq:sec_intermediates} is called an \term{input complex} of $Y$. A \term{single-input intermediate} has by definition a unique input complex. 
Note that there might be several possible choices of intermediates for a given network. One of the key ideas in the theory of intermediates is that some properties of the network  are preserved in a simpler \term{reduced network} defined as follows: 
one removes the intermediates, and all reactions involving intermediates, and adds a reaction $c\to c'$ for every sequence of reactions as \eqref{eq:sec_intermediates}
\cite{Feliu2013intermediates}.

By letting $x,y$ denote the vectors of concentrations of the species in $\X$ and $\Y$ respectively, the key idea   is that condition (i) ensures that, with mass-action kinetics, the entries of $y$ appear linearly in the ODE system \eqref{eq:ODE}. Then condition (ii) ensures that the steady state system has a unique zero in $y$, which in addition is a polynomial in $x$ with  coefficients being rational functions in $\k$ with all coefficients positive. 
Plugging the expressions of $y$ at steady state  into the remaining ODE equations (for $x$),  one obtains the ODE system for the reduced network for a choice of rate constants given as a vector of rational functions $\varphi(\k)$ in the original rate constants.
When  $Y_i$ is a single-input intermediate, the expression   takes the form $y_i=\psi_i(\k) x^{c}$, where $c$ is the vector of coefficients of the unique input of $Y_i$.  

As a convention, we order the species such that the vector of concentrations is $(x,y)$, that is,  
so that the non-intermediates come before the intermediates. We use a tilde to denote quantities and objects that correspond to the reduced network.

\begin{proposition}
\label{prop:intermediates_with_input_1_and_toricity}
For a reaction network consider a choice of intermediates $\S=\X\sqcup\Y$  with $\X=\{X_1,\ldots,X_n\}$ and $\Y=\{Y_1,\ldots,Y_\ell\}$ consisting of single-input intermediates. Let $F$ be the steady state system. 
Let $B\in\ZZ_{\geq 0}^{n\times \ell}$ be the matrix where the $i$th column is the coefficient vector in $\X$ of the unique input complex of the $i$th intermediate. Then the following holds:
\begin{enumerate}[label=(\roman*)]
\item There are rational maps $\psi\from\RR_{>0}^m\to\RR^\ell_{>0}$ and $\varphi\from\RR^m_{>0}\to\RR^{\tilde{m}}_{>0}$   such that we have a bijection
$$\Phi_{\k}\from
\VV_{>0}(\tilde{F}_{\varphi(\k)}) \to \VV_{>0}(F_\k),\quad x\mapsto (x,\psi(\k)\circ x^B).$$ 

\item
If $\VV_{>0}(F_\k)$ is $\T_A$-invariant over $\RR_{>0}$ for $A\in \ZZ^{d\times (n+\ell)}$, then $A=[ \tilde{A}\, | \, \tilde{A} B]$ with $\tilde{A}\in\ZZ^{d\times n}$.

\item With the notation in (ii), $\VV_{>0}(\tilde{F}_{\varphi(\k)})$ 
is $\T_{\tilde{A}}$-invariant over $\RR_{>0}$  if and only if 
$\VV_{>0}(F_\k)$ is $\T_A$-invariant over $\RR_{>0}$. Furthermore,  $\Phi_{\k}$ descends to a bijection $$\VV_{>0}(\tilde{F}_{\varphi(\k)})/\T_{\tilde{A}}^{>0}\to \VV_{>0}(F_\k)/\T_{A}^{>0}.$$

\item If $\tilde{F}$ is (generically/locally) $\T_{\tilde{A}}$-toric over $\RR_{>0}$,  then $F$ is (generically/locally) $\T_A$-toric over $\RR_{>0}$. The reverse implication holds also if  $\varphi$ is surjective. 
\end{enumerate}
\end{proposition}
\begin{proof}
Statement (i) is shown in \cite{Feliu2013intermediates}, see also \cite{sadeghimanesh2019grobner,sadeghimanesh:multi}. 
For statement (ii), 
let $y\in \VV_{>0}(F_\k)$ and write it  as $y=\Phi_\k(x)$ for the unique $x\in \VV_{>0}(\tilde{F}_{\varphi(\k)})$.
By writing $A=[\tilde{A}\, |\, \tilde{A}' ]$ with $\tilde{A}\in\ZZ^{d\times n}$, we have
\[\Phi_{\k}(x)\circ t^A = (x,\psi(\k)\circ x^B)\circ ( t^{\tilde{A}}, t^{\tilde{A}'})
=  (x\circ t^{\tilde{A}} ,\psi(\k)\circ x^B \circ t^{\tilde{A}'}).\]
By hypothesis, $\Phi_\k(x)\circ t^A \in \VV_{>0}(F_\k)$, 
hence it belongs to the image of $\Phi_\k$. Therefore
\[ \psi(\k)\circ x^B \circ t^{\tilde{A}'} = \psi(\k) \circ (x\circ t^{\tilde{A}})^B = 
\psi(\k) \circ x^B \circ t^{\tilde{A}B}, \]
and as this holds for all $t\in \RR^d_{>0}$, we must have that $\tilde{A}' = \tilde{A}B$, giving (ii).

Statement (iii) is now a   consequence of the equality $
\Phi_{\k}(x\circ t^{\tilde{A}})=
\Phi_{\k}(x)\circ t^A\,.$
Finally, (iv) follows from  (iii) as $\#(\VV_{>0}(\tilde{F}_{\varphi(\k)})/\T_{\tilde{A}}^{>0}) = \#(\VV_{>0}(F_\k)/\T_{A}^{>0})$. 
\end{proof}

\begin{remark}\label{rmk:input1}
    Surjectivity of $\varphi$ in \Cref{prop:intermediates_with_input_1_and_toricity}(i) corresponds to the realization condition  being satisfied for single-input intermediates by \cite[Prop.~5.3]{sadeghimanesh:multi}. 
In loc. cit. several scenarios where this holds are given. In particular, it holds in the common scenario where intermediates appear in isolated motifs of the form $c\ce{<->} Y_1 \ce{<->} \dots \ce{<->} Y_\ell \ce{->} c'$, with $\ce{<->}$ being either $\ce{->}$ or $\ce{<=>}$. 
We conjecture that   $\varphi$ is surjective whenever all intermediates are single-input.  
\end{remark}
\Cref{prop:intermediates_with_input_1_and_toricity} has important practical consequences: the reduced network is smaller as it has both fewer variables and reactions. Hence the computational cost for checking {whether the system is toric} is lower, sometimes dramatically lower. 
Additionally, it might be the case that some criteria are inconclusive for the original network, but succeed for the reduced network. An example of this is given below in \Cref{ex:shinar-feinberg}, where the injectivity test from \Cref{thm:injectivity} fails for the original network but is passed for the reduced network. 

\begin{example}\label{ex:idh_reduced}
For \Cref{fig:networks}(a), one possible choice of intermediates is   $\X=\{X_1,X_2,X_3,X_4\}$ and $\Y=\{X_5\}$, and the only input complex of $X_5$ is $X_3 + X_4$ (so the matrix $B$ in \Cref{prop:intermediates_with_input_1_and_toricity} is $(0 \ 0 \ 1\ 1)^{\top}$ and $\varphi$ is surjective, see \Cref{rmk:input1}). The  reduced network is
\[
X_1 + X_2 \ce{<=>} X_3 \ce{->} X_1 + X_4 \qquad 
    X_3 + X_4 \ce{->} X_2 + X_3\,.\]
The maps $\varphi$, $\psi$ and $\Phi_{\tilde{\k}}$ from \Cref{prop:intermediates_with_input_1_and_toricity}
are
\[\varphi(\k) =\left(\kappa_1,\kappa_2,\kappa_3,\tfrac{\kappa_4\kappa_6}{\kappa_5+\kappa_6}\right), \qquad \psi(\k)= \tfrac{\kappa_4}{\kappa_5+\kappa_6}, \qquad  
\Phi_{\k}(x_1,\dots,x_4) = \Big(x_1,\ldots,x_4,\tfrac{\kappa_4}{\kappa_5+\kappa_6}x_3x_4\Big). \]
\Cref{alg:summary} tells us that the  steady state system $\tilde{F}$ of the reduced network is $\T_{\tilde{A}}$-toric, hence 
the original steady state system in \Cref{ex:IDH} is $\T_A$-toric, with
\[\tilde{A}=\begin{bmatrix}1 & 0 & 1 & 0\\0 & 1 & 1 & 0\end{bmatrix}\qquad\text{and}\qquad A=[\ \tilde{A}\ |\ \tilde{A}B\ ] = \begin{bmatrix}1 & 0 & 1 & 0 & 1\\0 & 1 & 1 & 0 & 1\end{bmatrix}\,.\]
This is in accordance with what we saw in \Cref{ex:IDH}.
\end{example}

\begin{example}\label{ex:shinar-feinberg}
The classical network from Shinar and Feinberg's work on ACR \cite{Shinar2010structural} contains three single-input intermediates. The original and reduced networks are respectively:
\[
\begin{array}{ccc}
\begin{array}{c}
X_1 \ce{<=>} X_2 \ce{<=>} X_3 \ce{->} X_4\\
X_4 + X_5 \ce{<=>} X_6 \ce{->} X_2 + X_7\\
X_3 + X_7 \ce{<=>} X_8 \ce{->} X_3 + X_5\\
X_1 + X_7 \ce{<=>} X_9 \ce{->} X_1 + X_5
\end{array}
& \qquad
&
\begin{array}{c}
X_1 \ce{<=>} X_2 \ce{<=>} X_3 \ce{->} X_4\\
X_4 + X_5 \ce{->} X_2 + X_7\\
X_3 + X_7 \ce{->} X_3 + X_5\\
X_1 + X_7 \ce{->} X_1 + X_5\rlap{\,.}
\end{array}
\end{array}
\]
Applying \Cref{alg:invariance} to the reduced network, we conclude that 
the steady state systems of these two networks are torically invariant with respect to the following   matrices, respectively:
\[A=\left[\begin{matrix}1&1&1&0&1&1&0&1&1\\0&0&0&1&-1&0&0&0&0\end{matrix}\right]\qquad\text{and}\qquad \tilde{A}=\left[ \begin{matrix} 1&1&1&0&1&0\\ 0&0&0&1
&-1&0\end{matrix}\right].\]
\Cref{alg:summary} tells us that $\tilde{F}$ is $\T_{\tilde{A}}$-toric, as it passes the injectivity test, and \Cref{prop:intermediates_with_input_1_and_toricity} allows us to conclude that $F$ is $\T_A$-toric. However, the original system $F$ does not pass the injectivity test for {being $\T_A$-toric}. 
\end{example}

\subsection{Case study: Networks from ODEbase}
\label{subsec:odebase}

To illustrate the applicability of our results for realistic networks, we have applied our algorithms to the networks from the repository of biological and biomedical models {BioModels} \cite{biomodels}, using the stoichiometric matrices and reactant matrices collected in the database {ODEbase} \cite{odebase2022}. 

In our analysis, we work under the assumption of mass-action kinetics for all models (regardless of the exact kinetic model registered in BioModels).
We have considered all 69 nonlinear networks in ODEbase that satisfy
\[m\leq 100\,, \qquad  n-\rk(N)>0\,, \qquad  \text{and}\qquad \ker(N)\cap\mathbb{R}^m_{>0}\neq\varnothing\,.\]
For each such network we have applied \Cref{alg:summary} to the steady state system. If the network had single-input intermediates, we computed a matrix $\tilde{A}$ with invariance for the reduced network, and attempted to prove {that the reduced network was toric}, before proceeding with Step~\ref{line:injectivity_test}.

{\samepage The Github repository of this paper contains the output of the computations for each of the analyzed networks. We here report some summarized data:
\begin{itemize}
    \item For 38 networks, we rule out  {that the system is (locally) toric}. 
    \item For 31 networks, we verify that {the system is locally toric}, and for 30 of them, we verify {that it is toric}.
    \item For the remaining network {that is locally toric} (\texttt{835}), none of our conditions for {being toric} are satisfied, and the mixed volume bound is 46 (but the steady state ideal is binomial for all positive rate constants and hence {it is toric}). 
    \item Out of the toric networks, seven are not covered by the Deficiency One Theorem, and two are verified to be non-binomial.
    \item We verify capacity for multistationarity for 2 networks, and preclude multistationarity for 27 of them. We verify local ACR for 15 networks, and ACR for 14 of them.
    \item In total, 56 of the investigated networks are quasihomogeneous with respect to all weights in the toric invariance lattice. Out of these, 14 have a trivial matroid partition.
\end{itemize}
}

\subsection{Other flavors of {toric structure}}
\label{subsec:other_flavors}
In this final section, we view our results in the context of some previous approaches to determine {toric structures} in reaction networks. In this subsection, the reaction networks are taken with mass-action kinetics, and hence $M$ is the reactant matrix. 

\smallskip
\noindent
\textit{Quasithermostatic networks and deficiency theory.}
A special case {where the toric structure} plays an important role in classical monostationarity results is when the vectors in the toric invariance lattice span the left kernel of the stoichiometric matrix. 
A network with some choice of rate constants is said to be \term{quasithermostatic} if the set of positive steady states is of the form $x^*\circ\T_L^{>0}$ for some $x^*\in\RR^n_{>0}$ and a conservation law matrix $L$ as in \eqref{eq:classes} \cite[§4]{Horn1972general}. In this case, monostationarity follows directly from \Cref{prop:coset_counting}.

An important sufficient condition for quasithermostaticity is that the network is \emph{complex-balanced} (or a \emph{toric dynamical system} in the language of \cite{Craciun2009toric}). A characterization for when this happens for all choices of positive rate constants is given by the Deficiency Zero Theorem \cite{Horn1972}. Another sufficient condition for quasithermostaticity for all rate constants is given by the Deficiency One Theorem \cite{Feinberg1995class}.

To connect our results to this body of work, we review and reprove some basic facts about quasithermostatic networks from the point of view of partitions and toric invariance. Similar statements appear in \cite{Horn1972, Horn1972general, Horn1974aspects, Feinberg1995class} and more recently in \cite[§3.1]{Boros2013thesis}.
 
In what follows, we let $r$ be the number of complexes of a given reaction network and $m$ the number of reactions. The \term{linkage classes} of a network are the connected components of the network digraph which has the complexes as vertices and the reactions as edges. This results in a \term{linkage class partition} of $[m] = \gamma_1 \sqcup \dots \sqcup \gamma_\ell$,  where two indices are in the same subset if the corresponding reactions belong to the same linkage class. 
A network is \term{weakly reversible} if all linkage classes are strongly connected, and the \term{deficiency} of the network is 
$\delta :=  r  - s - \ell \geq 0$.

\begin{proposition}\label{prop:complex}
    Consider a network with $\ell$ linkage classes, its steady state system $F$, and $L\in \RR^{(n-s)\times n}$ a matrix defining the stoichiometric compatibility classes. 
    \begin{enumerate}[label=(\roman*)]
        \item If the matroid partition  of the steady state system   is finer than the linkage class partition, then $F$ is $\T_L$-invariant. 
        \item Statement (i) holds if the network is connected.
        \item Statement (i) holds if there is a direct sum decomposition $\im(N)=\im(N_1)\oplus\cdots\oplus \im(N_\ell)$, where $N_i$ is the stoichiometric matrix of the $i$th linkage class. 
        
        \item If the network is weakly reversible with deficiency zero or satisfies the conditions of the deficiency one theorem from \cite{Feinberg1995class}, then $F$ is locally $\T_L$-toric.
        
        \item If $F$ is $\T_L$-invariant, then for fixed $\k\in \RR^m_{>0}$, the cardinality of $\VV_{>0}(F_{\k},Lx-b)$ does not depend on $b$.  
        In particular $F$ is $\T_L$-toric over $\RR_{>0}$ if and only if $F$ is locally $\T_L$-toric and the network does not have the capacity for multistationarity. 
        
    \end{enumerate}
    \end{proposition}
\begin{proof}
Let $Y\in \RR^{n\times r}$ be the matrix whose columns are the coefficients of all complexes that appear in the network in some chosen order. The columns of $M$ are among the columns of $Y$. 
Let $C_G\in \ZZ^{r\times m}$ be the incidence matrix of the network seen as a directed graph: the entry $(i,j)$ is $1$, $-1$ or $0$, if the $i$th complex is on the right, left, or does not occur in the $j$th reaction, respectively.
Let $[m] = \gamma_1 \sqcup \dots \gamma_\ell$ be the  linkage class partition.
For each $k\in [\ell]$,  construct the vector $u_{k}\in \ZZ^{m}$ 
with 
$1$ for the indices   in $\gamma_k$   and  zero otherwise. 
 These vectors  generate $\ker(C_G^\top)$.

It is easy to see that $N=Y C_G$, hence $0 =L\,Y C_G$ by hypothesis, and the rows of $LY$ belong to the left-kernel of $C_G$. In particular columns of $L\,Y$ corresponding to   complexes
in the same linkage class are all equal.  It follows that a row $a$ of $L$
satisfies $a Y_i = a Y_j$ if $i,j\in \gamma_k$ for some $k$. Statement (i) now follows from \Cref{thm:characterization_of_invariance}, using that the matroid partition is finer than the linkage class partition. 

For (ii),  if the network is connected, then $\ell=1$, hence (i) holds. 
For (iii), the condition is equivalent to $\ker(N)=\ker(N_1)\oplus\cdots\oplus \ker(N_\ell)$. Hence, {a circuit of the column matroid of $N$
must be contained in a block} of the linkage class partition and (i) holds. 
For (iv), if the deficiency is zero, then $\ker(N)=\ker(C_G)$, see, e.g., \cite[Lem.~6.1.4]{Feinberg1995class}. After a suitable reordering of the complexes and reactions, $C_G$ is a block diagonal matrix, which gives that (iii) applies. 
Condition (iii) holds under the setting of the deficiency one theorem by hypothesis. 
We note also that the property in \Cref{prop:nondeg} holds for networks of deficiency zero and in the setting of the deficiency one theorem \cite[§15.2 and §17.1]{Feinberg2019foundations}, from where we conclude that $F$ is locally $\T_L$-toric.
For (v), we simply note that $(F_\k,Lx-b)$ is the coset counting system. 
\end{proof}

A natural generalization of complex-balancing is that the network gives rise to the same ODEs as a complex balanced network (which is called being \emph{disguised toric} in \cite{Brustenga2022disguised}). The set of rate constants for which the network is disguised toric is called the \emph{disguised toric locus}, and has rich dynamical and geometric properties (see, e.g., {\cite{BorosCraciunHenrikssonJinRojas2025} for a recent overview}). The methods developed in this paper give an easy-to-check \emph{necessary} condition for the disguised toric locus to have nonempty Euclidean interior, namely that \Cref{alg:invariance} returns a conservation law matrix $L$ and $n=s+d$. 

\begin{proposition}
If a network is disguised toric for rate constants in a nonempty Euclidean open set $U\subseteq\RR^m_{>0}$, then it is generically locally $\T_L$-toric,  where $L$ is a conservation law matrix. 
\end{proposition}

\begin{proof}
For each $\k\in U$, the network being disguised toric means that $\VV_{>0}(N(\k\circ x^M))$ is  $\T_{A_\k}$-toric for $A_\k$ a full rank matrix such that $A_\k N(\k\circ x^M)=0$ for all $x\in \RR^n_{>0}$. 
Furthermore, since  complex balanced networks have  positive steady states,  $U\subseteq\Z_{>0}$,  and it follows from \Cref{prop:vertical} that the steady state system is nondegenerate. Hence, there is a nonempty open subset $V\subseteq U$ of rate constants for which the network is both nondegenerate and disguised toric. 
Nondegeneracy implies now that $A_\k=L$ for all $\k\in V$  (see \cite[§3.4]{feliu:dimension}) and we have quasithermostaticity. By \Cref{thm:invariance_lifts_to_closure}, it follows that we have $\T_L$-invariance for all $\k\in\RR^m_{>0}$, and \Cref{thm:characterization_generic_local_toricity} then gives that {the system is generically locally $\T_L$-toric}.
\end{proof}

\smallskip
\noindent
\textit{Networks with binomial steady state ideals.}
By \Cref{rmk:binomial_implies_toric}, a sufficient condition for  a network {being generically  toric}  is that the polynomial ideal $\langle F_\kappa\rangle$ is binomial over the field $\QQ(\kappa)$ of rational functions in the rate constants. 
In \cite{Millan2012toricsteadystates},  networks with binomial steady state ideals are said to have \emph{toric steady states}, and the authors give a sufficient condition for this to hold for all rate constants. 
{In what follows, we will revisit this condition from the point of view of the techniques developed in the present paper. 
In particular, we will see that both works involve partitions, and that the one used in \cite{Millan2012toricsteadystates} can be seen as a coarsening of the matroid partition, which is consistent with the fact that binomiality of $\langle F_\k\rangle$ is a stronger property than $\VV_{>0}(F_\k)$ admitting a monomial parametrization (cf. \Cref{rmk:binomial_implies_toric}).}

We begin by writing the steady state system in the form
\[F_\k = \Sigma_\k x^Y,\]
where $\Sigma_\k \in \QQ(\k)^{n \times p}$ is the coefficient matrix and $Y\in \ZZ^{n\times p}$ has exactly one column per reactant complex of the network (the matrix $\Sigma_{\k}$ in   \cite{Millan2012toricsteadystates} might have some additional zero columns, but these are irrelevant for the results under discussion). Note that $m\geq p$.

Condition 3.1 in  \cite{Millan2012toricsteadystates} asks for the existence of a basis $b^1,\dots,b^d \in \RR^{p}_{\geq 0}$ for $\ker(\Sigma_{\k})$ such that their supports $I_1,\dots,I_d$ form a partition of $[p]$. 
When this is the case, the system {is toric} with respect to the maximal-rank matrix $A$ such that $AY_i=AY_j$ whenever $i,j$ belong to the same subset $I_k$ (see  \cite[Thm.~3.11]{Millan2012toricsteadystates}). This construction resembles \Cref{thm:characterization_of_invariance}. 
To understand the connection, we need first to assume that the partition is independent of   $\k\in \RR^m_{>0}$. 
Then,  the vectors $b^1,\dots,b^d$ are rational functions in $\k$, and by multiplying by the denominators if necessary, we can assume they are polynomial and hence continuous functions in $ \RR^m_{\geq 0}$.

Let $\iota\colon [m] \rightarrow [p]$ where $\iota(i)$ is the index of the column of $Y$ that has $M_i$ as column. 
The hypothesis of \cite{Millan2012toricsteadystates} gives then that
\begin{equation} \label{eq:AB}
AM_i=AM_j \quad \text{if $\iota(i),\iota(j)\in I_k$ for some $k\in [d]$}.
\end{equation} 
Note that $\iota$ is surjective and $\iota^{-1}$ induces a partition of $[m]$.
The connection between \cite{Millan2012toricsteadystates} and this work stems from the fact that the matroid partition  is finer than that induced by  $\iota^{-1}$, which we show next.

Let $K_{\k}\in \RR^{m\times p}$ be the matrix such that $\k\circ x^M = K_{\k}x^Y$, and more explicitly, $(K_{\k})_{i, \iota(i)}=\kappa_i$ with all other entries equal to zero. It follows that 
 $\Sigma_{\k} = N K_{\k}\,$, and that $K_\k$ has rank $p$ for all $\k\in \RR^m_{>0}$. 
By construction, it holds that $v_{\k,j}:=K_{\k}\, b^j \in \ker(N)$ for all $\k\in \RR^m_{>0}$, and    $\supp(v_{\k,j})=\iota^{-1}(I_j)$.
By continuity, if some entries of $\k$ are set to zero, the  vector  $v_{\k,j}$ still belongs to $\ker(N)$. 

As $K_{\k} (1,\dots,1)^\top = \k$, any vector in $\ker(N)$ belongs to $\im(K_{\k})$ for some $\k\in \RR^m$. 
Hence
\[\ker(N)  = \bigcup_{\k\in \RR^m}
\im(K_{\k}) \cap \ker(N)  =\bigcup_{\k\in \RR^m} K_{\k} ( \ker(\Sigma_{\k}))  \,, \] 
where in the last equality we use that $K_{\k}$ has maximal column rank. 
Using that $b^1,\dots,b^d$ form a basis for $\ker(\Sigma_{\k})$, and that the vectors 
$v_{\k,j}$, $v_{\k,i}$ have disjoint support if $i\neq j$,
we obtain that the support of any minimal-support vector in $\ker(N)$ is contained in one of $\supp(v_{\k,j})=\iota^{-1} (I_j)$. 
This implies that the matroid partition  is finer than that induced by  $\iota^{-1}$  as desired.

\begin{example}
\label{ex:trianglenetwork_revisited2}
Consider the triangle network from \Cref{fig:networks}(b), with steady state system \eqref{eq:triangle}. It relates to the various notions  discussed in this section in the following way:
\begin{itemize}
    \item This network has a binomial steady state ideal (or \emph{toric steady states}) if and only if $\kappa_1=\kappa_2$. {In particular, Condition 3.1 of \cite{Millan2012toricsteadystates} is not satisfied.}
    \item  It follows from the matrix--tree theorem \cite[§2]{Craciun2009toric} that it has complex-balanced steady states (gives rise to a \emph{toric dynamical system}) if and only if  $\kappa_1\kappa_3(\kappa_1\kappa_4+\kappa_2\kappa_4)=(\kappa_3\kappa_4)^2$.
    \item The network satisfies the conditions of the deficiency one theorem. It is also dynamically equivalent to a complex-balanced network (is \emph{disguised toric}) for all rate constants (cf. \cite[Thm.~3.1]{Brustenga2022disguised}). Each of these observations gives that the network is quasithermostatic for all $\k\in\RR^m_{>0}$.
\end{itemize}
These observations fit with what we have already seen in \Cref{ex:trianglenetwork_revisited}, namely that the network is toric with respect to $A=L$.
\end{example}


\newlength{\bibitemsep}\setlength{\bibitemsep}{.08\baselineskip plus .05\baselineskip minus .08\baselineskip}
\newlength{\bibparskip}\setlength{\bibparskip}{0pt}
\let\oldthebibliography\thebibliography
\renewcommand\thebibliography[1]{
  \oldthebibliography{#1}
  \setlength{\parskip}{2\bibitemsep}
  \setlength{\itemsep}{\bibparskip}
}

\def\bibfont{\footnotesize}

\bibliographystyle{alpha} 
\newcommand{\etalchar}[1]{$^{#1}$}

\vspace{1em}

\filbreak 

{\small
\noindent \textsc{Elisenda Feliu}\\
\textsc{University of Copenhagen}\\
\url{efeliu@math.ku.dk}\\[0.5em]
\textsc{Oskar Henriksson}\\
\textsc{University of Copenhagen}\\
\textsc{Current address: Center for Systems Biology Dresden}\\
\textsc{Max Planck Institute of Molecular Cell Biology and Genetics}\\
\url{oskar.henriksson@mpi-cbg.de}
}

\end{document}